\documentclass[12pt,reqno,a4paper]{amsart}
\usepackage{euscript}
\usepackage{amssymb}
\usepackage{breqn}
\usepackage{graphicx}
\usepackage{color}
\usepackage{subfigure,graphicx}
\usepackage{overpic}
\renewcommand{\epsilon}{\varepsilon}
\usepackage{float}
\usepackage{soul}

\newcommand{\R}{\ensuremath{\mathbb{R}}}

\newcommand{\Z}{\ensuremath{\mathbb{Z}}}

\newcommand{\CC}{\mathcal{C}}

\newcommand{\CO}{\ensuremath{\mathcal{O}}}

\newcommand{\ov}{\overline}

\newcommand{\T}{\theta}
\newcommand{\f}{\varphi}
\newcommand{\al}{\alpha}
\newcommand{\be}{\beta}

\newcommand{\s}{\ensuremath{\mathbb{S}}}
\newcommand{\CA}{\ensuremath{\mathcal{A}}}
\newcommand{\CB}{\ensuremath{\mathcal{B}}}
\newcommand{\C}{\ensuremath{\mathcal{C}}}

\newcommand{\U}{\ensuremath{\mathcal{U}}}

\newcommand{\bx}{\mathbf{x}}
\newcommand{\by}{\mathbf{y}}
\newcommand{\bz}{\mathbf{z}}

\newcommand{\de}{\delta}

\newcommand{\bxi}{\boldsymbol \xi}

\newcommand{\bg}{{\bf g}}
\newcommand{\bF}{{\bf F}}
\def\p{\partial}
\def\e{\varepsilon}

\newtheorem {theorem} {Theorem} %[section]
\newtheorem {proposition} {Proposition}

\newtheorem {lemma} {Lemma}

\newtheorem {remark} {Remark}

\newtheorem {mtheorem} {Theorem}

\begin{document}
\allowdisplaybreaks
\title[On the torus bifurcation in averaging theory]{On the torus bifurcation in averaging theory}

\begin{abstract}
In this paper, we take advantage of the averaging theory to investigate a torus bifurcation in two-parameter families of $2D$ nonautonomous differential equations. Our strategy consists in looking for generic conditions on the averaged functions that ensure the existence of a curve in the parameter space characterized by a Neimark-Sacker bifurcation in the corresponding Poincar\'{e} map. A Neimark-Sacker bifurcation for planar maps consists in the birth of an invariant closed curve from a fixed point, as the fixed point changes stability. In addition, we apply our results to study a torus bifurcation in a family of $3D$ vector fields.
\end{abstract}

\author{Murilo R. C\^{a}ndido and Douglas D. Novaes}

\address{$^1$ Departamento de Matem\'{a}tica, Universidade
Estadual de Campinas, Rua S\'{e}rgio Baruque de Holanda, 651, Cidade Universit\'{a}ria Zeferino Vaz, 13083-859, Campinas, SP,
Brazil} 
\email{candidomr@ime.unicamp.br}
\email{ddnovaes@ime.unicamp.br}

\keywords{averaging theory, invariant torus, periodic solution, Hopf bifurcation, Naimark-Sacker bifurcation}

\subjclass[2010]{Primary: 34C23, 34C29, 34C45}

%34C23->Bifurcation
%34C29->Averaging
%34C45->Invariant Manifolds

\maketitle

\section{Introduction  and statements of the main result}
In the present study, we consider two-parameter families of nonautonomous differential equations given by
\begin{equation}\label{tm1}
\dot \bx= \e \bF_1(t,\bx;\mu)+\e^2 \widetilde{\bf F}(t,\bx;\mu,\e).
\end{equation}
Here, $\bF_1$ and $\widetilde\bF$ are $\C^1$ functions $T$-periodic in the variable $t\in\R,$  $\bx=(x,y)\in\Omega$ with $\Omega$ an open bounded subset of $\mathbb{R}^2,$ $\e \in (-\e_0,\e_0)$ for some $\e_0>0$ small, and $\mu \in \R.$  Throughout in this paper, we shall consider the differential equation \eqref{tm1} defined in the extended phase space $\s^1\times \Omega$ by taking $\dot t=1,$ where $\s^1=\R/T \Z.$

Detecting limit sets of differential equations is a major problem in the qualitative theory of dynamical systems. In particular, there are several research pieces  dealing with the existence of limit cycles for differential equations of kind \eqref{tm1}. In this direction, the {\it averaging theory} (see \cite{SVM} and \cite[Chapter $11$]{V}) is one of the most used methods. In short, this theory provides a sequence of functions $\bg_i,$ $i=1,2,\ldots,k,$ each one called {\it $i$-th order averaged function}, which ``control'' the bifurcation of isolated periodic solutions of \eqref{tm1} (see Appendix). The first averaged function can be defined as
$$
\bg_1(\bx;\mu)=\big(\bg_1^1(\bx;\mu),\bg_1^2(\bx;\mu)\big)=\int_0^T\bF_1(t,\bx;\mu)dt,
$$
which provides the so-called {\it averaged system}
\begin{equation}\label{tas}
\dot{\bx}=\e\bg_1(\bx;\mu).
\end{equation}

When dealing with invariant sets of differential equations, the Poincar\'e map $P:\Sigma\rightarrow\Sigma$ is a classic tool in understanding their properties.  It is defined in a transversal section of the orbits $\Sigma,$ which leads to a dimensional reduction of the problem. In turn, this provides conceptual clarity for many notions that are somewhat cumbersome to state for differential equations. For instance, a periodic orbit of a differential equation corresponds to a fixed point of a Poincar\'{e} map and, consequently, the notion of orbital stability is reduced to the stability of a fixed point (see \cite[Chapter 10]{wiggins}). As another example, an invariant torus of a differential equation corresponds to an invariant closed curve $\gamma\subset\Sigma$ of a Poincar\'{e} map $P,$ that is, $P(\gamma)=\gamma.$

No general method exists for constructing Poincar\'e maps of arbitrary differential equations. Nevertheless, if a nonautonomous $T$-periodic differential equation $\dot\bx=F(t,\bx)$ admits, for each initial condition $\bx \in\Omega,$ a unique solution $\f(t,\bx)$ defined for every $t\in[0,T]$ and satisfying $\f(0,\bx)=\bx,$ then the Poincar\'{e} map defined in the tranversal section $\Sigma=\{0\}\times\Omega$ is given by $P(\bx)=\f(T,\bx).$ Indeed, since $F$ is $T$-periodic in the time variable $t,$ we can take $\dot t=1$ and consider the differential equation defined in the extended phase space $\s^1\times\Omega.$ In this way, $P$ maps $\Sigma$ onto itself, which is identified with $\Omega.$ When $F$ is given as \eqref{tm1}, Lemma \ref{lap} from Appendix (see \cite[Lemma $1$]{LNT}) provides the Taylor expansion  of the Poincar\'{e} map $P(\bx)=P(\bx;\mu,\e)$ around $\e=0.$ The coefficients of this expansion determine the averaged functions $\bg_i$'s.

One can find results in research literature corelating the existence of invariant tori of the differential equation \eqref{tm1} with Hopf bifurcation of the averaged system \eqref{tas}. This fact is briefly commented on \cite[Appendix $C.5$]{SVM}. Similar results can be found in  \cite[Section $4.C$]{lhopf} and \cite[Chapter $2$]{thesis}.

The main goal of this paper is to provide generic conditions on the averaged functions $\bg_i$'s
to guarantee the existence of a codimension-one bifurcation curve $\mu(\e)$ in the parameter space $(\mu,\e)$ characterized by the birth of an invariant torus of \eqref{tm1}  from a periodic solution. 

Our strategy consists in looking for conditions that ensure the existence of a {\it Neimark-Sacker Bifurcation} (see \cite{N,S1,S2}) in the Poincar\'{e} map of \eqref{tm1}. In discrete dynamical system theory, a Neimark-Sacker bifurcation in a one-parameter family of planar maps is characterized by the birth of an invariant closed curve from a fixed point, as the fixed point changes stability.
Since invariant closed curves of Poincar\'{e} maps correspond to  invariant tori of differential equations, a Neimark-Sacker bifurcation in Poincar\'{e} maps is called  {\it Torus Bifurcation}. We shall discuss this bifurcation in Section \ref{NSB}.

\subsection{Setting of the problem} 
In what follows, we shall assume that $(\bx_{\mu_0};\mu_0)$ is a {\it Hopf point} of the averaged system \eqref{tas}, that is, $\bg_1(\bx_{\mu_0};\mu_0)=0$ and  the Jacobian matrix $D_\bx\bg_1(\bx_{\mu_0};\mu_0)$ has a pair of conjugated purely imaginary eigenvalues $\pm i\omega_0$ $(\omega_0>0).$ By the Implicit Function Theorem, there exists a continuous curve $\mu\mapsto\bx_\mu \in \Omega,$ defined in an interval $J\ni \mu_0,$ such that $\bg_1(\bx_\mu;\mu)=0$ for every $\mu\in J.$ Clearly, the pair of complex conjugated eigenvalues  $\al(\mu)\pm i \beta(\mu)$ of the Jacobian matrix $D_\bx\bg_1(\bx_\mu;\mu)$ satisfies $\al(\mu_0)=0$ and $\beta(\mu_0)=\omega_0>0.$  Through a linear change of variables, we can assume that $D_\bx\bg_1(\bx_{\mu_0};0)$ is in its real Jordan normal form. Thus, the existence of a Hopf point $(x_{\mu_0};\mu_0)$ is equivalent to the following assumption.
\begin{itemize}
\item[{\bf A1.}]  {\it There exists a continuous curve $\mu\in J \mapsto\bx_\mu \in \Omega,$ defined in an interval $J\ni \mu_0,$ such that $\bg_1(\bx_\mu;\mu)=0$ for every $\mu\in J\subset\R,$  the pair of complex conjugated eigenvalues  $\al(\mu)\pm i \beta(\mu)$ of $D_\bx\bg_1(\bx_\mu;\mu)$  satisfies $\al(\mu_0)=0$ and $\beta(\mu_0)=\omega_0,$ and  $D_\bx\bg_1(\bx_{\mu_0};0)$ is in its real Jordan normal form.}
\end{itemize}

As a first consequence of hypothesis {\bf A1}, the next lemma ensures the existence of a periodic solution of the differential equation \eqref{tm1}.

\begin{lemma}\label{l1}
 Assume that hypothesis {\bf A1} holds. Then, there exist a neighborhood $J_0\subset J$ of $\mu_0$ and $\e_1,$ $0<\e_1<\e_0$ such that, for every $(\mu,\e)\in J_0\times(-\e_1,\e_1),$ the differential equation \eqref{tm1} admits a unique $T$-periodic solution $\f(t;\mu,\e)$ satisfying $\f(0;\mu,\e)\to \bx_{\mu}$ as $\e\to0.$
\end{lemma}

Lemma \ref{l1} is proven in Section \ref{PTA}.  We notice that, when  the differential equation \eqref{tm1} is defined in the extended phase space $\s^1\times\Omega,$  such a periodic solution is given by $\Phi(t;\mu,\e)=(t,\f(t;\mu,\e)).$

We also assume the following traversal hypothesis.

\begin{itemize}
\item[{\bf A2.}] {\it  Let $\al(\mu)\pm i \beta(\mu)$ be the pair of complex conjugated eigenvalues of $D_\bx\bg_1(\bx_\mu;\mu)$ such that $\al(\mu_0)=0,$ $\beta(\mu_0)=\omega_0>0.$ Assume that $\al'(\mu_0)\neq0$.}
\end{itemize}

Finally, define the number
\begin{equation}\label{l1TA}
\begin{array}{rl}
\ell_{1,1}&=\dfrac{1}{8}\left(\dfrac{\partial ^3\bg_1^1(\bx_{\mu_0};\mu_0)}{\partial x^3}+\dfrac{\partial ^3\bg_1^1(\bx_{\mu_0};\mu_0)}{\partial x\partial y^2}+\dfrac{\partial ^3\bg_1^2(\bx_{\mu_0};\mu_0)}{\partial x^2\partial y}+\dfrac{\partial ^3\bg_1^2(\bx_{\mu_0};\mu_0)}{\partial y^3}\right)\\
&+\dfrac{1}{8\omega_0}\left(\dfrac{\partial ^2\bg_1^1(\bx_{\mu_0};\mu_0)}{\partial x\partial y}\Big(\dfrac{\partial ^2\bg_1^1(\bx_{\mu_0};\mu_0)}{\partial x^2}+\dfrac{\partial ^2\bg_1^1(\bx_{\mu_0};\mu_0)}{\partial y^2}\Big)-\dfrac{\partial ^2\bg_1^2(\bx_{\mu_0};\mu_0)}{\partial x\partial y}\right.\\
&\Big(\dfrac{\partial ^2\bg_1^2(\bx_{\mu_0};\mu_0)}{\partial x^2}+\dfrac{\partial ^2\bg_1^2(\bx_{\mu_0};\mu_0)}{\partial y^2}\Big)-\dfrac{\partial ^2\bg_1^1(\bx_{\mu_0};\mu_0)}{\partial x^2}\dfrac{\partial ^2\bg_1^2(\bx_{\mu_0};\mu_0)}{\partial x^2}\\
&\left.+\dfrac{\partial ^2\bg_1^1(\bx_{\mu_0};\mu_0)}{\partial y^2}\dfrac{\partial ^2\bg_1^2(\bx_{\mu_0};\mu_0)}{\partial y^2}\right).
\end{array}
\end{equation}
It is worth mentioning that $\e \,\ell_{1,1}$ is the first Lyapunov coefficient of the averaged system \eqref{tas} at  $(\bx_{\mu_0};\mu_0).$ Thus, {\bf A1}, {\bf A2}, and $\ell_{1,1}\neq0$ characterize a {\it Hopf Bifurcation} in the averaged system \eqref{tas} (see  \cite{gh83,k}).

\subsection{First order torus bifurcation} As our first main result, we establish the relation between a {\it Hopf Bifurcation} in the averaged system \eqref{tas} with a {\it Torus Bifurcation}  in the differential equation \eqref{tm1}.

\begin{mtheorem}\label{teo1}In addition to hypotheses {\bf A1} and {\bf A2}, assume that  $\ell_{1,1}\neq0.$ Then, for each $\e>0$ sufficiently small, there exist a $C^1$ curve $\mu(\e)\in J_0,$ with  $\mu(0)=\mu_0,$ and neighborhoods $\U_{\e}\subset \s^1\times\Omega$ of the periodic solution $\Phi(t;\mu(\e),\e)$ and  $J_{\e}\subset J_0$ of $\mu(\e)$ for which the following statements hold.
\begin{itemize}

\item[$(i)$] For $\mu\in J_{\e}$ such that $\ell_{1,1}(\mu-\mu(\e))\geq0,$ the periodic orbit $\Phi(t;\mu(\e),\e)$ is unstable (resp. asymptotically stable), provided that $\ell_{1,1}>0$ (resp. $\ell_{1,1}<0$), and the differential equation \eqref{tm1} does not admit any invariant tori in $\U_{\e}.$

\item[$(ii)$] For $\mu\in J_{\e}$ such that $\ell_{1,1}(\mu-\mu(\e))<0,$ the differential equation \eqref{tm1} admits a unique invariant torus $T_{\mu,\e}$ in $\U_{\e}$ surrounding the periodic orbit $\Phi(t;\mu,\e).$ Moreover,  $T_{\mu,\e}$ is unstable (resp. asymptotically stable), whereas the periodic orbit $\Phi(t;\mu,\e)$ is asymptotically stable (resp. unstable), provided that $\ell_{1,1}>0$ (resp. $\ell_{1,1}<0$). 

\item[$(iii)$] $T_{\mu,\e}$ is the unique invariant torus of the differential equation \eqref{tm1} bifurcating from the periodic orbit $\Phi(t;\mu(\e),\e)$ in $\U_{\e}$ as $\mu$ passes through $\mu(\e).$
\end{itemize}

\end{mtheorem}

\begin{remark}
The $\C^1$ differentiability of the functions $\bF_1$ and $\widetilde\bF$ was the very first assumption on the differential equation \eqref{tm1}. It is worth mentioning that this hypothesis is not strictly necessary in order to apply Theorem \ref{teo1}. In fact, we shall see that it is sufficient to have the differentiability of the ``Poincar\'{e} Map''  of the differential equation \eqref{tm1}. This implies that Theorem \ref{tm1} can be applied to a wider class of differential equations, in particular for the class of piecewise smooth differential equation introduced in \cite{llinovrod}.
\end{remark}

\subsection{Structure of the paper} In Section \ref{NSB}, we discuss the Neimark-Sacker bifurcation, which plays a key whole in the proof of  Theorem \ref{teo1}. Section \ref{PTA} is devoted to the proof of  Theorem \ref{teo1}. Afterward, in Section \ref{HOA}, we state Theorem \ref{teo2}, which generalizes Theorem \ref{teo1} by establishing weaker conditions on the higher order averaged functions $\bg_i$
still ensuring a torus bifurcation. In Section \ref{HPl}, we relate Hopf bifurcation in the higher order averaged system to a torus bifurcation in the corresponding differential equation.
 Finally, in Section \ref{Ex}, the obtained results are applied to study a torus bifurcation in a family of 3D vector fields. An Appendix is provided with the formulae of the averaged functions.

\section{Neimark-Sacker Bifurcation}\label{NSB}
The proof of our main result is mainly based on the classical Neimark-Sacker Bifurcation, which is a version of Hopf Bifurcation for maps. In what follows we shall briefly discuss this bifurcation.

Consider the following one parameter family of maps
\begin{equation}\label{mf1}
\bx \mapsto \bF(\bx;\sigma), \quad \bx=(x_1,x_2)^\intercal\in \R^2, \quad \sigma \in \R^1.
\end{equation}
Assume that $\bx=0$ is a fixed point of the map \eqref{mf1},  for every $|\sigma|$ sufficiently small.
Denote by $r(\sigma)e^{\pm i \varphi(\sigma)}$ the pair of complex conjugated eigenvalues of the Jacobian matrix $D_{\bx}\bF(0,\sigma).$  We shall assume that $r(0)=1$ and $\varphi(0)=\theta,$ with $0<\theta<\pi.$  Also, consider the Taylor expansion of $\bF(\bx;0)$ around $\bx=0$ as 
\begin{equation*}\label{mf2}
\bF(\bx;0)=A \bx+\dfrac{1}{2}B(\bx,\bx)+\dfrac{1}{6}C(\bx,\bx,\bx)+\CO(||\bx||^4),
\end{equation*}
where $B(\bx,\by)=\big(B^1(\bx,\by),B^2(\bx,\by)\big)$ and $C(\bx,\by,\bz)=\big(C^1(\bx,\by,\bz),$ $C^2(\bx,\by,\bz)\big)$ are multilinear functions with the following components
\begin{equation}\label{comp}
\begin{array}{l}
\displaystyle B^i(\bx,\by)=\sum_{j,k=1}^2\dfrac{\partial^2\bF_i}{\partial x_j\partial x_k}(0;0)\, x_jy_k,\vspace{0.2cm}\\

\displaystyle C^i(\bx,\by,\bz)=\sum_{j,k,l=1}^2\dfrac{\partial^3\bF_i}{\partial x_j\partial x_k\partial x_l}(0;0)\, x_jy_kz_l,
\end{array}
\end{equation}
for $i=1,2,$ and  $A=D_{\bx} \bF(0,0).$  

We use the elements above to construct the Lyapunov coefficient $\ell_1$ of the map \eqref{mf1} at $(\bx;\sigma)=(0;0).$ Accordingly, let ${\bf p},{\bf q} \in \mathbb{C}^2$ be, respectively,  complex eigenvectors of $A^\intercal$ and $A$  satisfying $A^\intercal{\bf p}=e^{-i \theta}{\bf p},$ $A{\bf q}=e^{ i \theta}{\bf q},$ and $\langle {\bf p},{\bf q}\rangle=1.$ Here, for $u,v\in \mathbb{C}^2,$ we are considering the inner product $\langle u,v\rangle=\ov{u}^\intercal \cdot v.$ Thus, we define
\begin{equation}\label{ell1}
\ell_1:=\textrm{Re}\left(\dfrac{e^{-i\theta}g_{21}}{2}\right)-\textrm{Re}\left(\dfrac{(1-2e^{i\theta})e^{-2i\theta}}{2(1-e^{i\theta})}g_{20}g_{11}\right)-\dfrac{1}{2}|g_{11}|^2-\dfrac{1}{4}|g_{02}|^2\neq0,
\end{equation}
where 
\begin{equation}\label{gij}
\begin{array}{l}
g_{21}=\langle {\bf p},C({\bf q},{\bf q},\ov {\bf q})\rangle,\,\, g_{20}=\langle {\bf p},B({\bf q},{\bf q})\rangle,\vspace{0,2cm}\\
g_{11}=\langle {\bf p},B({\bf q},\ov {\bf q})\rangle, \text{ and }  \,\, g_{02}=\langle {\bf p},B(\ov {\bf q},\ov {\bf q})\rangle.
\end{array}
\end{equation}

Under generic conditions, a Neimark-Sacker bifurcation is characterized by the existence of a neighborhood of the fixed point $\bx=0$ in which a unique invariant closed curve bifurcates from $\bx=0$ (see {\cite[Theorem $4.6$]{k}}). The next theorem provides generic conditions ensuring a Neimark-Sacker bifurcation in \eqref{mf1}.

\begin{theorem}\label{l5}
Suppose that for $|\sigma|$ sufficiently small  $\bx=0$ is a fixed point of the map \eqref{mf1} with complex eigenvalues $r(\sigma)e^{\pm i \f(\sigma)}$ satisfying $r(0)=1$ and $\f(0)=\T,$  $0<\theta<\pi.$ In addition, assume that
\begin{itemize}
\item[$C.1$] $r'(0)\neq 0,$
\item[$C.2$]$e^{ik\theta}\neq 1,$ for $k=1,2,3,4,$ and
\item[$C.3$] $\ell_1\neq0.$
\end{itemize}
Then, there exists neighborhoods $U\subset\R^2$ of $\bx=0$ and $I\subset \R$ of $\sigma=0$ for which the following statements hold.
\begin{itemize}
\item[$(i)$] For  $\sigma\in I$ such that $\ell_1\sigma\geq 0,$ the fixed point $\bx=0$ is unstable (resp. asymptotically stable), provided that $\ell_1>0$ (resp. $\ell_1<0$), and the map \eqref{mf1} does not admit any invariant closed curve in $U.$  
\item[$(ii)$] For  $\sigma\in I$ such that $\ell_1\sigma<0,$  the map \eqref{mf1} admits a unique invariant closed curve $S_{\mu}$ in $U$ surrounding the fixed point $\bx=0.$  Moreover, $S_{\mu}$ is unstable (resp. asymptotically stable), whereas the fixed point $\bx=0$ is asymptotically stable (resp. unstable), provided that $\ell_1>0$ (resp. $\ell_1<0).$
\item[$(iii)$] $S_{\mu}$ is the unique invariant closed curve of the map \eqref{mf1} bifurcating from the fixed point $\bx=0$ in $U$ as $\sigma$ pass through $0.$
\end{itemize}
\end{theorem}

\section{Proofs of Lemma \ref{l1} and Theorem \ref{teo1}}\label{PTA}
 The {\it Poincar\'{e} Map} of the differential equation \eqref{tm1}, defined on the section $\Sigma=\{0\}\times \Omega,$ writes 
\begin{equation}\label{frm}
\bx\mapsto P(\bx;\mu,\e)=\bx+\e \bg_1(\bx;\mu)+\e^2\widetilde G(\bx;\mu,\e).
\end{equation}
In what follows, we shall prove Lemma \ref{l1} by showing the existence of fixed points $\xi(\mu,\e)$ for the Poincar\'{e} Map.

\begin{proof}[Proof of Lemma \ref{l1}]
Define 
\[
f(\bx,\mu,\e):=\dfrac{P(\bx;\mu,\e)-\bx}{\e}=\bg_1(\bx;\mu)+\e\widetilde G(\bx;\mu,\e).
\]
Notice that $f(\bx_{\mu_0},\mu_0,0)=(0,0)$ and
\[
\dfrac{\p f}{\p \bx}(\bx_\mu,\mu,0)= \p_{\bx} \bg_1(\bx_\mu;\mu).
\] 
From the hypothesis {\bf A1}, $\al(\mu_0)=0$ and $\be(\mu_0)=\omega_0\neq0,$ where $\al(\mu)\pm i\beta(\mu)$ are the complex conjugated eigenvalues of  $\p_{\bx} \bg_1(\bx_\mu;\mu).$ Therefore, there exists a neighborhood $J_0\subset J$ of $\mu_0$ such that $\beta(\mu)\neq0$ for every $\mu\in \ov{J_0}.$ Consequently,
\[
\left|\dfrac{\p f}{\p \bx}(\bx_\mu,\mu,0)\right|\neq 0,
\]
for every $\mu\in \ov{J_0}.$ Hence, from the Implicit Function Theorem and from the compactness of $\ov{J_0},$ there exists $\e_1,$ $0<\e_1<\e_0,$ and a unique function $\bxi(\mu,\e),$ defined on $\ov{J_0}\times(-\e_1,\e_1),$ such that $\bxi(\mu,0)=\bx_\mu$ and $f(\bxi(\mu,\e),\e)=0$ for every $\e\in(-\e_1,\e_1)$ and $\mu\in\ov{J_0}.$ \end{proof}

The next result provides a curve $\mu(\e)$ of critical values for the parameter $\mu$ regarding the fixed point $\bxi(\mu,\e)$ of the map \eqref{frm} for which the conditions of Theorem \ref{l5} hold.

\begin{lemma}\label{l2}
For each $(\mu,\e)\in J_0\times (-\e_1,\e_1),$ let $\lambda(\mu,\e)$ and  $\ov{\lambda(\mu,\e)}$ be the pair of complex conjugated eigenvalues of $D_{\bx} P (\bxi(\mu,\e);\mu,\e)$ and assume that hypotheses {\bf A1} and {\bf A2} hold. Then, there exists $\e_2,$ $0<\e_2<\e_1,$ and a unique smooth function $\mu:(-\e_2,\e_2)\rightarrow J_0,$ with  $\mu(0)=\mu_0,$ satisfying
\begin{itemize}
\item[{\bf 1.}] $|\lambda(\mu(\e),\e)|=1,$ 

\smallskip

\item[{\bf 2.}] $\big(\lambda(\mu({\e}),\e)\big)^k\neq 1,$ for $k\in\{1,2,3,4\},$ and
\item[{\bf 3.}] $\dfrac{d}{d\mu} |\lambda(\mu,\e)|\Big|_{\mu=\mu({\e})}\neq 0,$
\end{itemize}
for every $\e\in(-\e_2,\e_2)\setminus\{0\}.$
\end{lemma}

\begin{proof}

For each $(\mu,\e)\in J_0\times (-\e_1,\e_1),$ the Jacobian matrix of the first return map $P(\bx;\mu,\e)$ at its fixed point $\bxi(\mu,\e)$ is given by
$$
D_\bx P (\bxi(\mu,\e);\mu,\e)=Id+\e \dfrac{\p \bg_1}{\p\bx}(\bx_\mu;\mu)+\CO(\e^2),
$$
which has the following eigenvalues
\begin{equation*}\label{ev}
\begin{array}{l}
\lambda(\mu,\e)=1+ \e\left(\al(\mu)+ i\beta(\mu) \right)+\CO(\e^2), \text{ and}\vspace{0.2cm}\\
 \ov{\lambda(\mu,\e)}=1+ \e\left(\al(\mu)- i\beta(\mu) \right)+\CO(\e^2).
 \end{array}
\end{equation*}
Notice that
\begin{equation}\label{lambda2}
\begin{array}{rl}
|\lambda(\mu,\e)|^2&=1+2\e   \al(\mu)+\CO(\e^2)\vspace{0.1cm}\\
&=1+\e\, \ell(\mu,\e),
 \end{array}
 \end{equation}
where $\ell(\mu,\e)=2\al(\mu)+\CO(\e).$ From hypothesis {\bf A2}, we have 
$$
\ell(\mu_0,0)=0 \quad \mbox{and} \quad \dfrac{\p \ell}{\p \mu}(\mu_0,0)=2\al'(\mu_0)\neq 0.
$$
Thus, by the Implicit Function Theorem,  there exist $\e_2,$ $0<\e_2<\e_1,$ and a unique function $\mu:(-\e_2,\e_2)\rightarrow J_0=(\mu_0-\de_1,\mu_0+\de_1)$ such that $\mu(0)=\mu_0$ and $\ell(\mu(\e),\e)=0,$ for every $\e\in(-\e_2,\e_2).$  This implies that $|\lambda(\mu(\e),\e)|=1,$ for every $\e\in(-\e_2,\e_2).$ Hence, statement {\bf 1} is proved. Moreover,
since
\begin{equation}\label{lam_e}
\begin{array}{rl}
\lambda(\mu(\e),\e)=&1+\e\big(\al(\mu(\e))+ i\beta(\mu(\e))\big)+\CO(\e^2)\\
=&1+ \e\left(i   \omega_0\right)+\CO(\e^2),
\end{array}
\end{equation}
and $\omega_0>0,$ the parameter $\e_2>0$ can be made smaller, if necessary, in order that
$$
 \lambda(\mu(\e),\e)\not\in \left\lbrace \pm 1, \pm i,-\dfrac{1}{2}\pm i \dfrac{\sqrt{3}}{2}\right\rbrace,
 $$ 
 for every $\e\in(-\e_2,\e_2)\setminus\{0\}.$  Consequently, for $k\in\{1,2,3,4\},$ $\left(\lambda(\mu(\e),\e)\right)^k\neq 1$ for every $\e\in(-\e_2,\e_2)\setminus\{0\},$ which proves statement ${\bf 2}.$
Finally, computing the derivative of \eqref{lambda2} at $\mu=\mu_0$ we implicitly obtain that
\[
 \dfrac{\p}{\p \mu} | \lambda(\mu,\e)|\Big|_{\mu=\mu(\e)}=\alpha'(\mu({\e}))\e+\CO(\e^2)=\alpha'(\mu_0)\e+\CO(\e^2).
\]
Since $\al'(\mu_0)\neq0,$ the parameter $\e_2>0$ can be made smaller again, if necessary, in order that
\[
 \dfrac{\p}{\p \mu} | \lambda(\mu,\e)|\Big|_{\mu=\mu(\e)}\neq0,
\] 
for every $\e\in(-\e_2,\e_2)\setminus\{0\}.$ This concludes the proof of statement ${\bf 3}.$
\end{proof}

Now, we are ready to prove Theorem \ref{teo1}.

\begin{proof}[Proof of Theorem \ref{teo1}] For each $(\mu,\e)\in J_0\times (-\e_2,\e_2),$
let $\bxi(\mu,\e)$ be the fixed point of the Poincar\'e map \eqref{frm} given by Lemma \ref{l1} and let $\mu(\e)$ be the curve of critical values of the parameter $\mu$ given by Lemma \ref{l2}. 

Changing the coordinates  in \eqref{frm} by setting $\bx=\by + \bxi(\mu,\e)$ and taking  $\mu=\sigma+\mu(\e),$ we get the map
\begin{equation}\label{0frm2}
\by\to {\bf H}_\e(\by;\sigma):=P(\by+\bxi(\sigma+\mu(\e),\e);\sigma+\mu(\e),\e)-\bxi(\sigma+\mu(\e),\e).
\end{equation}
Notice that
\[
{\bf H}_\e(\by;\sigma)=\by+\e \bg_1(\by+\bxi(\sigma+\mu(\e),\e),\sigma+\mu(\e))+\e^2\widetilde G(\by+\bxi(\sigma+\mu(\e),\e),\e;\sigma+\mu(\e)).
\]
The proof of Theorem \ref{teo1} will follow by showing that, for each $\e>0$ sufficiently small, the map \eqref{0frm2} satisfies the hypotheses of Theorem \ref{l5}.

First of all, fix $\e_3$ (to be chosen later on) satisfying $0<\e_3<\e_2.$ Denote by $I'_{\e}$ the set of $\sigma\in\R$ such that $\sigma+\mu(\e)\in J_0.$ Since, from Lemma \ref{l2}, $\mu(\e)\in J_0,$ we get that $0\in I'_{\e}.$  Thus, for each $\e\in(0,\e_3),$ Lemma \ref{l1} implies that $\by=0$ is a fixed point of \eqref{0frm2} for every $\sigma\in I'_{\e}.$ Notice that $\eta_{\e}(\sigma)=\lambda(\sigma+\mu(\e),\e)$ and $\ov {\eta_{\e}(\sigma)}=\ov{\lambda(\sigma+\mu(\e),\e)}$ are the eigenvalues of $D_{\by} {\bf H}_\e(0;\sigma).$ Denote $\eta_{\e}(\sigma)=r_{\e}(\sigma)e^{ i \varphi_{\e}(\sigma)}$ and $\varphi_{\e}(0)=\T_{\e},$ $0<\T_{\e}<\pi.$ Thus, from Lemma \ref{l2}, we get
\[
r_{\e}(0)=1,\,\, e^{ik\T_{\e}}\neq 1,\,\, \text{for}\,\, k\in\{1,2,3,4\},\,\, and\,\, r'_{\e}(0)\neq 0,
\]
for every $\e\in(0,\e_3).$ Therefore, the map \eqref{0frm2} satisfies all the conditions of Theorem \ref{l5} but $C.3$ for each $\e\in(0,\e_3).$ 

In order to check condition $C.3,$ we need to compute $\ell_1$ as defined in \eqref{ell1}. Following the procedure of Section \ref{NSB}, we first compute the Taylor expansion of ${\bf H}_\e(\by;0)$ around $\by=0$ as
$$
{\bf H}_\e(\by;0)= A_{\e} \by+\dfrac{1}{2}B_{\e}(\by,\by)+\dfrac{1}{6}C_{\e}(\by,\by,\by)+\CO(||\by||^4),
$$
where $B_{\e}({\bf u},{\bf v})=\left(B^1_{\e}({\bf u},{\bf v}),B^2_{\e}({\bf u},{\bf v})\right)$ and $C_{\e}({\bf u},{\bf v},{\bf w})=\big(C^1_{\e}({\bf u},{\bf v},{\bf w}),$ $C^2_{\e}({\bf u},{\bf v},{\bf w})\big)$ are multilinear functions with the following components
\begin{equation}\label{comp}
\begin{array}{rl}
 B^i_{\e}({\bf u},{\bf v})=&\!\!\displaystyle\e \sum_{j,k=1}^2\dfrac{\partial^2\bg_1^i}{\partial x_j\partial x_k}(\bxi(\mu(\e),\e);\mu(\e))\,u_jv_k+\CO(\e^2)\vspace{0.2cm}\\
=&\!\!\displaystyle\e \sum_{j,k=1}^2\dfrac{\partial^2\bg_1^i}{\partial x_j\partial x_k}(\bx_{\mu_0};\mu_0)\,u_jv_k+\CO(\e^2),\vspace{0.2cm}\\

 C^i_{\e}({\bf u},{\bf v},{\bf w})=&\!\!\displaystyle \e \sum_{j,k,l=1}^2\dfrac{\partial^3 \bg_1^i}{\partial x_j\partial x_k\partial x_l}(\bxi(\mu(\e),\e);\mu(\e))\,u_jv_kw_l+\CO(\e^2)\vspace{0.2cm}\\
 
 =&\!\!\displaystyle \e \sum_{j,k,l=1}^2\dfrac{\partial^3 \bg_1^i}{\partial x_j\partial x_k\partial x_l}(\bx_{\mu_0};\mu_0)\,u_jv_kw_l+\CO(\e^2),
\end{array}
\end{equation}
for $i=1,2,$ and  
\begin{equation}\label{Ae}
\begin{array}{rl}
A_{\e}=&\!\!D_{\bx} H_{\e}(0;0)=\mathrm{Id}+\e\,D_\bx \bg_1(\bxi(\mu(\e),\e);\mu(\e))+\CO(\e^2)\vspace{0.2cm}\\
=&\!\!\mathrm{Id}+\e\,D_\bx \bg_1(\bx_{\mu_0};\mu_0)+\CO(\e^2).
\end{array}
\end{equation}
Now, let ${\bf p}_{\e}\in \mathbb{C}^2$ and ${\bf q}_{\e}\in \mathbb{C}^2$  be, respectively,  the eigenvectors of the matrices $A_{\e}^\intercal$ and $A_{\e}$ satisfying $A_{\e}^\intercal{\bf p}_{\e}=e^{-i \theta_{\e}}{\bf p}_{\e},$ $A_{\e}{\bf q}_{\e}=e^{ i \theta_{\e}}{\bf q}_{\e},$ and $\langle {\bf p}_{\e},{\bf q}_{\e}\rangle=1.$ We claim that ${\bf p}_{\e}={\bf p} +\CO(\e)$ and ${\bf q}_{\e}={\bf q} +\CO(\e),$ where ${\bf p},{\bf q}\in \mathbb{C}^2$ satisfy $D_\bx \bg_1(\bx_{\mu_0};\mu_0)^\intercal {\bf p}=-i\omega_0 {\bf p},$  $D_\bx \bg_1(\bx_{\mu_0};\mu_0) {\bf q}=i\omega_0 {\bf q},$ and $\langle {\bf p},{\bf q}\rangle=1.$ Indeed, 
an  eigenvector $\by\in \mathbb{C}^2$ of $A_{\e}$ with respect to $\eta_{\e}(0)$ satisfies 
\begin{equation}\label{fv}
\left[A_{\e}-\eta_{\e}(0)\,\mathrm{Id}\right]\by=0.
\end{equation}
From  \eqref{lam_e}, \eqref{Ae}, and taking $\by=\by_0+\CO(\e),$  equation \eqref{fv} writes
$$
\left[D_{\bx} \bg_1(\bx_{\mu_0};\mu_0)-i\omega_0 \,\mathrm{Id}\right]\by_0 +\CO(\e)=0.
$$
Matching the coefficients of $\e,$ we get that $\by_0$ is an eigenvector of $D_{\bx} \bg_1(\bx_{\mu_0};\mu_0)$ with respect to the eigenvalue $i\omega_0.$ We can do the same for the matrix $A_{\e}^\intercal.$ Finally, since $\langle {\bf p}_{\e},{\bf q}_{\e}\rangle=1$ for every $\e,$ we conclude that $\langle {\bf p},{\bf q}\rangle=1.$ Furthermore, from hypothesis {\bf A1}, $D_\bx\bg_1(\bx_{\mu_0};\mu_0)$ is in its real normal Jordan form, thus we can take ${\bf p}={\bf q}=(1,-i)/\sqrt{2}.$

Defining
\begin{equation}\label{mult0}
B_{0}({\bf u},{\bf v})=\dfrac{d}{d\e}B_{\e}({\bf u},{\bf v})\Big|_{\e=0} \quad \text{and}\quad {C}_{0}({\bf q},{\bf q},\ov{\bf q})=\dfrac{d}{d\e}{C}_{\e}({\bf q},{\bf q},\ov{\bf q})\Big|_{\e=0},
\end{equation}
and denoting
$g_{20}^0=\langle {\bf p}, B_{0}({\bf q},{\bf q})\rangle,$ $g_{11}^0=\langle {\bf p}, B_{0}({\bf q},\ov{\bf q}) \rangle,$  $g_{02}^0= \langle {\bf p}, B_{0}(\ov{\bf q},\ov{\bf q}) \rangle,$ and $g_{21}^0=\langle {\bf p},    C_{0}({\bf q},{\bf q},\ov {\bf q })\rangle,$ we get, from \eqref{gij} and \eqref{comp}, that
\[
g_{20}=\langle {\bf p}_{\e}, B_{\e}({\bf q}_{\e},{\bf q}_{\e}) \rangle
=\langle {\bf p}, B_{\e}({\bf q},{\bf q}) \rangle+\CO(\e^2)
=\e\,g_{20}^0+\CO(\e^2).
\]
Analogously,  $g_{11}=\e\,g_{11}^0+\CO(\e^2),$  $g_{02}=\e\, g_{02}^0+\CO(\e^2),$ and $g_{21}=\e\,g_{21}^0+\CO(\e^2).$ 
From \eqref{lam_e}, $e^{i\theta_{\e}}=\eta_{\e}(0)=\lambda(\mu(\e),\e)=1+\e (i \omega_0)+\CO(\e^2),$ thus
\begin{align*}
&\dfrac{e^{-i\theta_\e}g_{21}}{2}=\e\dfrac{g_{21}^0}{2}+\CO(\e^2)\quad\text{and}\\
&\dfrac{(1-2e^{i\theta_{\e}})e^{-2i\theta_{\e}}}{2(1-e^{i\theta_{\e}})}g_{20}g_{11}=-\e \dfrac{i}{2\omega_0}g_{20}^0 g_{11}^0+\CO(\e^2).
\end{align*}
Hence, substituting the above expressions into \eqref{ell1}, we obtain
\begin{equation}\label{ell2}
\ell_1=\dfrac{\e}{2}\left(\textrm{Re}(g_{21}^0)-\dfrac{\textrm{Re}(i\,g_{20}^0 g_{11}^0)}{\omega_0}\right)+\CO(\e^2).
\end{equation}
Moreover, from \eqref{comp} and \eqref{mult0}, we compute
\begin{align*}
\textrm{Re}(g_{21}^0)=&\dfrac{1}{4}\left(\dfrac{\partial ^3\bg_1^1(\bx_{\mu_0};\mu_0)}{\partial x^3}+\dfrac{\partial ^3\bg_1^1(\bx_{\mu_0};\mu_0)}{\partial x\partial y^2}+\dfrac{\partial ^3\bg_1^2(\bx_{\mu_0};\mu_0)}{\partial x^2\partial y}+\dfrac{\partial ^3\bg_1^2(\bx_{\mu_0};\mu_0)}{\partial y^3}\right),\nonumber\\
\textrm{Re}(i\,g_{20}^0 g_{11}^0)=&\dfrac{1}{4}\left(\dfrac{\partial ^2\bg_1^2(\bx_{\mu_0};\mu_0)}{\partial x\partial y}\Big(\dfrac{\partial ^2\bg_1^2(\bx_{\mu_0};\mu_0)}{\partial x^2}+\dfrac{\partial ^2\bg_1^2(\bx_{\mu_0};\mu_0)}{\partial y^2}\Big)-\dfrac{\partial ^2\bg_1^1(\bx_{\mu_0};\mu_0)}{\partial x\partial y}\right.\nonumber\\
&\cdot\Big(\dfrac{\partial ^2\bg_1^1(\bx_{\mu_0};\mu_0)}{\partial x^2}+\dfrac{\partial ^2\bg_1^1(\bx_{\mu_0};\mu_0)}{\partial y^2}\Big)+\dfrac{\partial ^2\bg_1^1(\bx_{\mu_0};\mu_0)}{\partial x^2}\dfrac{\partial ^2\bg_1^2(\bx_{\mu_0};\mu_0)}{\partial x^2}\nonumber\\
&\left.-\dfrac{\partial ^2\bg_1^1(\bx_{\mu_0};\mu_0)}{\partial y^2}\dfrac{\partial ^2\bg_1^2(\bx_{\mu_0};\mu_0)}{\partial y^2}\right).
\end{align*}
Substituting the above expressions  into \eqref{ell2} we conclude that
\begin{align*}
\ell_1=\e \, \ell_{1,1}+\CO(\e^2),
\end{align*}
where $\ell_{1,1}$ is given by \eqref{l1TA}. From hypothesis, $\ell_{1,1}\neq 0.$ Therefore, we can choose $\e_3>0,$ $0<\e_3<\e_2,$ in order that $\mathrm{sgn}(\ell_1)=\mathrm{sgn}(\ell_{1,1})$ for every $\e\in(0,\e_3).$ 

Then, for each $\e\in(0,\e_3),$ applying Theorem \ref{l5} for the map $H_{\e},$  we get the existence of neighborhoods $U_{\e}\subset\R^2$ of the fixed point $\by=0$ and $I_{\e}\subset I'_{\e}$ of the critical parameter $\sigma=0$ for which items $(i),$ $(ii),$ and $(iii)$ of Theorem \ref{l5} holds. Going back through the change of variables and parameters we get the existence of  neighborhoods $U'_{\e}\subset \Omega$ of the fixed point $\bx=\bxi(\mu(\e),\e)$ and  $J_{\e}\subset J_0$ of the critical parameter $\mu=\mu(\e)$ for which the following statements hold.
\begin{itemize}
\item[$(i)$] For $\mu\in J_{\e}$ such that $\ell_{1,1}(\mu-\mu(\e))\geq0,$ 
the fixed point $\bx=\bxi(\mu(\e),\e)$ is unstable (resp. asymptotically stable), provided that $\ell_1>0$ (resp. $\ell_1<0$), and the Poincar\'{e} map \eqref{frm} does not admit any invariant closed curve in $U'_{\e}.$ 

\item[$(ii)$] For $\mu\in J_{\e}$ such that $\ell_{1,1}(\mu-\mu(\e))<0,$ the Poincar\'{e} map \eqref{frm} admits a unique invariant closed curve $S_{\mu,\e}$ in $U'_{\e}$ surrounding the fixed point $\bxi(\mu,\e).$  Moreover, $S_{\mu,\e}$ is unstable (resp. asymptotically stable), whereas the fixed point $\bx=\bxi(\mu,\e)$ is asymptotically stable (resp. unstable), provided that $\ell_1>0$ (resp. $\ell_1<0).$

\item[$(iii)$] $S_{\mu,\e}$ is the unique invariant closed curve of the Poincar\'{e} map  \eqref{frm} bifurcating from the fixed point $\bx=\bxi(\mu(\e),\e)$ in $U'_{\e}$ as $\mu$ pass through $\mu(\e).$
\end{itemize}

  Finally, define $\U_{\e}$ as the saturation of $\{0\}\times U'_{\e}$ through $\Phi$, that is, $\U_{\e}=\{\Phi(t,\bx;\mu,\e):t\in[0,T],\bx\in U'_{\e}\}.$ Hence, the proof of Theorem \ref{teo1} follows by noticing that the  invariant closed curve $S_{\mu,\e}$ in $U'_{\e}$ of the Poincar\'{e} map \eqref{frm} surrounding the fixed point $\bxi(\mu,\e)$ corresponds to an invariant torus $T_{\mu,\e}$ in $\U_{\e}$ of the differential equation \eqref{tm1} (defined in the extended phase space $\s^1\times \Omega$) surrounding the periodic orbit $\Phi(t;\mu,\e).$
 \end{proof}

\section{Higher order approach}\label{HOA}
In this section, we consider two-parameter families of nonautonomous differential equations given by
\begin{equation}\label{s1}
\dot\bx(t)=\sum_{i=1}^k\e^i \bF_i(t,\bx;\mu)+\e^{k+1} \widetilde{\bF}(t,\bx;\mu,\e).
\end{equation}
Here, $\bF_i,$ $i=1,2,\ldots,k,$ and $\widetilde\bF$ are  sufficiently smooth functions and $T$-periodic in the variable $t\in\R,$  $\bx=(x,y)\in\Omega$ with $\Omega$ an open bounded subset of $\mathbb{R}^2,$  $\e \in (-\e_0,\e_0)$ for some $\e_0>0$ small, and $\mu\in\R.$ 

In what follows we shall apply the same ideas of the previous section for obtaining a higher order version of Theorem \ref{teo1}.

\subsection{Setting of the problem} Consider the averaged functions $\bg_i,$ $i=1,2,\ldots,k,$ as defined in Appendix. Let $l,$ $1\leq l<k,$ be the subindex of the first non-vanishing averaging function.  From Lemma \ref{lap} of the Appendix,  the {\it Poincar\'{e} Map} of the differential equation \eqref{s1}, defined on the transversal section $\Sigma=\{0\}\times \Omega,$ writes 
\begin{equation}\label{frm2}
\bx\mapsto P(\bx;\mu,\e)=\bx+\e^l {\bf G}(\bx, \mu, \e),
\end{equation}
where $$
{\bf G}(\bx, (\mu, \e))=\bg_l(\bx;\mu)+\e^1 \bg_{l+1}(\bx;\mu)+\dots+ \e^{k-l}\bg_{k}(\bx;\mu)+\e^{k+1-l}\widetilde{\bf G}(\bx;\mu,\e)
.$$

As a first hypothesis we assume that the averaged system of order $l$
\begin{equation}\label{tasl}
\dot \bx=\e^l \bg_l(\bx;\mu)
\end{equation}
has a Hopf point at $(\bx_{\mu_0};\mu_0).$ Equivalently, suppose that
\begin{itemize}
\item[{\bf B1.}] {\it there exists a continuous curve $\mu\in J\mapsto\bx_\mu \in \Omega,$ defined in an interval $J\ni \mu_0,$ such that $\bg_l(\bx_\mu;\mu)=0$ for every $\mu \in J$ and  the pair of complex conjugated eigenvalues $\al(\mu)\pm i \beta(\mu)$ of  $D_\bx\bg_l(\bx_\mu;\mu)$ satisfies $\al(\mu_0)=0$ and $\beta(\mu_0)=\omega_0>0.$  }
\end{itemize} 

Here, the proof of Lemma \ref{l1} can be followed straightly in order to get a neighborhood $J_0\subset J$ of $\mu_0,$  a parameter $\e_1,$ $0<\e_1<\e_0,$ and a unique function $\bxi: J_0\times(-\e_1,\e_1)\rightarrow \R^2$ satisfying $\bxi(\mu,0)=\bx_{\mu}$ and $P(\bxi(\mu,\e);\mu,\e)=\bxi(\mu,\e),$ for every $(\mu,\e)\in J_0\times (-\e_1,\e_1).$ This provides the following lemma.

\begin{lemma}\label{l1higher}
 Assume that hypothesis {\bf B1} holds. Then, there exists a neighborhood $J_0\subset J$ of $\mu_0$ and $\e_1,$ $0<\e_1<\e_0$ such that, for every $(\mu,\e)\in J_0\times(-\e_1,\e_1)$ the differential equation \eqref{s1} admits a unique $T$-periodic orbit $\f(t;\mu,\e)$ satisfying $\f(0;\mu,\e)\to \bx_{\mu}$ as $\e\to0.$
\end{lemma}

We notice that, when  the differential equation \eqref{s1} is defined in the extended phase space $\s^1\times\Omega,$  such a periodic solution is given by $\Phi(t;\mu,\e)=(t,\f(t;\mu,\e)).$
 
We also assume the following transversal hypothesis.
\begin{itemize}
\item[{\bf B2.}] {\it Let $\al(\mu)\pm i \beta(\mu)$ be the pair of complex conjugated eigenvalues of $D_\bx\bg_l (\bx_\mu;\mu)$ such that $\al(\mu_0)=0,$ $\beta(\mu_0)=\omega_0>0.$ Assume that
$$
\dfrac{d\al(\mu)}{d\mu}\Big|_{\mu=\mu_0}=d\neq 0.
$$} 
\end{itemize} 

The proof of Lemma \ref{l2} can also be followed directly in order to get the following result.

\begin{lemma}\label{l2higher}
For each $(\mu,\e)\in J_0\times (-\e_1,\e_1),$ let $\lambda(\mu,\e)$ and  $\ov{\lambda(\mu,\e)}$ be the pair of complex conjugated eigenvalues of $D_{\bx} P (\bxi(\mu,\e);\mu,\e)$ and assume that hypotheses {\bf B1} and {\bf B2} hold. Then, there exists $\e_2,$ $0<\e_2<\e_1,$ and a unique smooth function $\mu:(-\e_2,\e_2)\rightarrow J_0,$ with  $\mu(0)=\mu_0,$ satisfying
\smallskip
\begin{itemize}
\item[{\bf 1.}] $|\lambda(\mu(\e),\e)|=1,$ 

\item[{\bf 2.}] $\big(\lambda(\mu({\e}),\e)\big)^k\neq 1,$ for $k\in\{1,2,3,4\},$ and

\item[{\bf 3.}] $\dfrac{d}{d\mu} |\lambda(\mu,\e)|\Big|_{\mu=\mu({\e})}\neq 0.$
\end{itemize}
\end{lemma}
 We emphasize  that the functions $\bxi(\mu,\e)$ and $\mu(\e)$ can be both explicitly expanded in Taylor series around $\e=0$ up to order $\e^{k}.$ Due to the complexity of the coefficients of these expansions, we shall omit them here. 

Now, applying the change of variables $\bx=\by+\bxi(\mu,\e)$ and taking $\mu=\sigma+\mu(\e),$ the Poincar\'{e} map \eqref{frm2} writes
\begin{equation}\label{sish2}
\by \mapsto {\bf H}_\e(\by;\sigma):=\by+\e^l {\bf G}\left(\by+\bxi\left(\sigma+\mu(\e),\e\right), \sigma+\mu(\e),\e\right).
\end{equation}

Now, for each $\e\in(0,\e_2),$ denote by $I'_{\e}$ the set of $\sigma\in\R$ such that $\sigma+\mu(\e)\in J_0.$ Since, from Lemma \ref{l2higher}, $\mu(\e)\in J_0,$ we get that $0\in I'_{\e}.$  Thus, for each $\e\in(0,\e_2),$ Lemma \ref{l1higher} implies that $\by=0$ is a fixed point of \eqref{sish2} for every $\sigma\in I'_{\e}.$ Notice that $\eta_{\e}(\sigma)=\lambda(\sigma+\mu(\e),\e)$ and $\ov {\eta_{\e}(\sigma)}=\ov{\lambda(\sigma+\mu(\e),\e)}$ are the eigenvalues of $D_{\by} {\bf H}_\e(0;\sigma).$ Denote $\eta_{\e}(\sigma)=r_{\e}(\sigma)e^{ i \varphi_{\e}(\sigma)}$ and $\varphi_{\e}(0)=\T_{\e},$ $0<\T_{\e}<\pi.$ Thus, from Lemma \ref{l2higher}, we get
\[
r_{\e}(0)=1,\,\, e^{ik\T_{\e}}\neq 1,\,\, \text{for}\,\, k\in\{1,2,3,4\},\,\, and\,\, r'_{\e}(0)\neq 0,
\]
for every $\e\in(0,\e_2).$ Therefore, the map \eqref{sish2} satisfies all the conditions of Theorem \ref{l5} but $C.3$ for each $\e\in(0,\e_2).$  

Theorem \ref{teo1} will be generalized by showing that the Poincar\'{e} map \eqref{sish2} admits a Neimark-Sacker bifurcation. From here, it remains to show that condition $C.3$ holds. Accordingly, we need to compute $\ell_1$ as defined in \eqref{ell1}. Following the procedure of Section \ref{NSB}, we first compute the Taylor expansion of ${\bf H}_\e(\by;0)$ around $\by=0$ as
$$
{\bf H}_{\e}(\by,0)= A_\e\by+\dfrac{1}{2}B_\e(\by,\by)+\dfrac{1}{6} C_\e(\by,\by,\by)+\CO(|\by|^4),
$$
where  $B_\e({\bf u},{\bf v})$ and $C_\e({\bf u},{\bf v},{\bf w})$ are the multilinear functions defined in \eqref{comp}, and $A=D_{\bx} {\bf H}_{\e}(0,0).$  Moreover,
 $A_\e=Id+\e^l \CA_{\e}+\CO(\e^{k+1}),$ $B_\e=\e^l\CB_{\e} +\CO(\e^{k+1}),$  $C_\e=\e^l\CC_{\e}+\CO(\e^{k+1}),$ where
\begin{equation}\label{TeABC}
\begin{array}{l}
\CA_{\e}= A_l+\e A_{l+1}+\dots+\e^{k-l} A_k,\vspace{0.2cm}\\
\CB_{\e}({\bf u},{\bf v})= B_l({\bf u},{\bf v})+\e B_{l+1}({\bf u},{\bf v})+\dots+\e^{k-l} B_{k}({\bf u},{\bf v}),\vspace{0.2cm}\\
\CC_{\e}({\bf u},{\bf v},{\bf w})= C_l({\bf u},{\bf v},{\bf w})+\e C_{l+1}({\bf u},{\bf v},{\bf w})+\dots+\e^{k-l} C_{k}({\bf u},{\bf v},{\bf w}).
\end{array}
\end{equation}
We stress that $\CA_{\e},$ $\CB_{\e},$ and $\CC_{\e}$ can be explicitly computed. Due to the complexity of these expressions, we shall omit them here.

Through a linear change of variables  in \eqref{sish2}, we can assume,  without loss of generality, that
\begin{itemize}
\item[${\bf B3.}$] {\it for each $\e\in(0,\e_3),$ the  matrix  $Id+\e^l\CA_{\e}$ is in its real Jordan normal form.}
\end{itemize} 
This implies that
\begin{equation*}\label{dgj}
Id+\e^l\CA_{\e}=\begin{pmatrix}
1+\widetilde \al_\e& -\widetilde \beta_\e\\
&\\
\widetilde \beta_\e &1+ \widetilde \al_\e
\end{pmatrix},
\end{equation*}
where $ \widetilde \al_\e+i \widetilde\beta_\e$ and $\widetilde \al_\e-i \widetilde\beta_\e$ are the eigenvalues of $\CA_{\e}.$ Considering
\begin{equation}\label{aue}
e^{i \theta_\e}=\eta_{\e}(0)=1+\sum^{k}_{j=l}\e^{j}(\al_j+i \beta_j)+\CO(\e^{k+1}),
\end{equation} with $\al_l=0$ and $\beta_l=\omega_0>0,$
we have that
\[
\widetilde \al_\e= \sum^{k}_{j=l}\e^{j}\al_j, \quad\text{and}\quad \widetilde \beta_\e= \sum^{k}_{j=l}\e^{j}\beta_j.
\]

Finally, for ${\bf p}=(1,-i)/\sqrt{2}\in \mathbb{C}^2,$ define the number 
\begin{align*}
\ell_1^{\e}=&-\textrm{Re}\left(\dfrac{(1-2 e^{i\theta_\e})e^{-2i\theta_\e}}{2(1-e^{i\theta_\e})}\left\langle {\bf p}, \e^l\CB_\e({\bf p},{\bf p}) \right\rangle \left\langle {\bf p}, \e^l\CB_\e({\bf p},\ov {\bf p}) \right\rangle\right) \\
&+\textrm{Re}\left(e^{-i\theta_\e}\left\langle {\bf p}, \e^l\CC_\e({\bf p}, {\bf p}, \ov {\bf p})\right\rangle\right) -\dfrac{|\left\langle {\bf p}, \e^l\CB_\e({\bf p},\ov{\bf p})\right\rangle|^2}{2}-\dfrac{|\left\langle {\bf p}, \e^l\CB_\e(\ov{\bf p},\ov{\bf p})\right\rangle|^2}{4},
\end{align*}
and consider its Taylor expansion around $\e=0,$ which can be explicitly computed as
\begin{equation}\label{Tel1}
\ell_1^{\e}=\e^l\ell_{1,l}+\e^{l+1}\ell_{1,l+1}+\e^{l+2}\ell_{1,l+2}+\cdots+\e^{k}\ell_{1,k}+\CO(\e^{k+1}).
\end{equation}
In Section \ref{SLC} we provide the expressions of $\ell_{1,j},$ $l\leq j\leq k.$

\subsection{Higher order torus bifurcation} As our second main result, Theorem \ref{teo1} is generalized as follows.

\begin{mtheorem}\label{teo2}
Let $l,$ $1\leq l\leq k,$ be the subindex of the first non-vanishing averaging function and $\ell_{1,j},$ $l\leq j\leq k,$ as defined in \eqref{Tel1}. In addition to hypotheses {\bf B1}, {\bf B2}, and {\bf B3}, assume that $\ell_{1,j}\neq0$ for some $l\leq j\leq k.$ Let $j^*,$ $l\leq j^*\leq k,$ be the first subindex such that $\ell_{1,j^*}\neq0.$ Then, for each $\e>0$ sufficiently small there exist a $C^1$ curve $\mu(\e)\in J_0,$ with  $\mu(0)=\mu_0,$ and neighborhoods $\U_{\e}\subset \s^1\times \Omega$ of the periodic solution $\Phi(t;\mu(\e),\e)$ and  $J_{\e}\subset J_0$ of $\mu(\e)$ for which the following statements hold.\begin{itemize}

\item[$(i)$] For $\mu\in J_{\e}$ such that $\ell_{1,j^*}(\mu-\mu(\e))\geq0,$ the periodic orbit $\Phi(t;\mu(\e),\e)$ is unstable (resp. asymptotically stable), provided that $\ell_{1,j^*}>0$ (resp. $\ell_{1,j^*}<0$), and the differential equation \eqref{s1} does not admit any invariant tori in $\U_{\e}.$

\item[$(ii)$] For $\mu\in J_{\e}$ such that $\ell_{1,j^*}(\mu-\mu(\e))<0,$ the differential equation \eqref{s1} admits a unique invariant torus $T_{\mu,\e}$ in $\U_{\e}$ surrounding the periodic orbit $\Phi(t;\mu,\e).$ Moreover,  $T_{\mu,\e}$ is unstable (resp. asymptotically stable), whereas the periodic orbit $\Phi(t;\mu,\e)$ is asymptotically stable (resp. unstable), provided that $\ell_{1,j^*}>0$ (resp. $\ell_{1,j^*}<0$). 

\item[$(iii)$] $T_{\mu,\e}$ is the unique invariant torus of the differential equation \eqref{s1} bifurcating from the periodic orbit $\Phi(t;\mu(\e),\e)$ in $\U_{\e}$ when $\mu$ passes through $\mu(\e).$
\end{itemize}
\end{mtheorem}

\begin{proof}
As before, fix $\e_3$ (to be chosen later on) satisfying $0<\e_3<\e_2.$ We already have that for each $\e\in(0,\e_3)$ the map \eqref{0frm2} satisfies all the conditions of Theorem \ref{l5} but $C.3.$ 

In order to check condition $C.3$ we need to compute $\ell_1$ as defined in \eqref{ell1}. Accordingly, let ${\bf p}_{\e}\in \mathbb{C}^2$ and ${\bf q}_{\e}\in \mathbb{C}^2$ be, respectively, eigenvectors of $A_{\e}^\intercal$ and $A_{\e}$ satisfying $A_{\e}{\bf q}_{\e}=e^{ i \theta_{\e}}{\bf q}_{\e},$  $A_{\e}^\intercal{\bf p}_{\e}=e^{-i \theta_{\e}}{\bf p}_{\e},$ and $\langle {\bf p}_{\e},{\bf q}_{\e}\rangle=1.$ Analogous to the proof of Theorem \ref{teo1},  ${\bf p}_{\e}=\widetilde {\bf p}_{\e} +\CO(\e^{k-l+1})$ and ${\bf q}_{\e}=\widetilde {\bf q}_{\e}  +\CO(\e^{k-l+1}),$ where $\widetilde {\bf p}_{\e},\widetilde {\bf q}_{\e} \in \mathbb{C}^2$ satisfy $\CA_{\e}^\intercal \widetilde {\bf p}_{\e}=(\widetilde \al_\e-i \widetilde\beta_\e) \widetilde {\bf p}_{\e},$  $\CA_{\e}\widetilde {\bf q}_{\e}=(\widetilde \al_\e+i \widetilde\beta_\e)  \widetilde {\bf q}_{\e},$ and $\langle \widetilde {\bf p}_{\e},\widetilde {\bf q}_{\e}\rangle=1.$ Furthermore, from hypothesis {\bf B3}, $\CA_{\e}$ is in its normal Jordan form, thus we can take $\widetilde {\bf p}_{\e}=\widetilde {\bf q}_{\e}={\bf p}:=(1,-i)/\sqrt{2}.$ Hence,
\begin{align}\label{gijhigher}
\begin{split}
g_{20}=&\left\langle {\bf p}_{\e}, B_\e({\bf q}_{\e}, {\bf q}_{\e})\right\rangle=\left\langle {\bf p}, \e^l\CB_{\e}({\bf p}, {\bf p})\right\rangle+\CO(\e^{k+1}),\vspace{0.2cm}\\
g_{11}=&\left\langle {\bf p}_{\e}, B_\e({\bf q}_{\e}, {\bf \ov q}_{\e})\right\rangle=\left\langle {\bf p}, \e^l\CB_{\e}({\bf p}, {\bf \ov p})\right\rangle+\CO(\e^{k+1}),\vspace{0.2cm}\\
g_{02}=&\left\langle {\bf p}_{\e}, B_\e({\bf \ov q}_{\e}, {\bf \ov q}_{\e})\right\rangle=\left\langle {\bf p}, \e^l\CB_{\e}({\bf \ov p}, {\bf \ov p})\right\rangle+\CO(\e^{k+1}),\vspace{0.2cm}\\
g_{21}=&\left\langle {\bf p}_{\e}, C_\e({\bf  q}_{\e}, {\bf  q}_{\e}, {\bf \ov q}_{\e})\right\rangle=\left\langle {\bf p}, \e^l\CC_{\e}({\bf  p}, {\bf  p}, {\bf \ov p})\right\rangle+\CO(\e^{k+1}).
\end{split}
\end{align} 
Now, from \eqref{ell1}, we compute
\begin{equation}\label{a0e}
\ell_1= \ell_1^{\e} +\CO(\e^{k+1}),
\end{equation} 
where $\ell_1^{\e}$ is given by \eqref{Tel1}. From hypothesis,  $j^*,$ $l\leq j^*\leq k,$ is the first subindex such that $\ell_{1,j^*}\neq0.$ Therefore, we can choose $\e_3>0,$ $0<\e_3<\e_2,$ in order that $\mathrm{sgn}(\ell_1)=\mathrm{sgn}(\ell_{1,j^*})$ for every $\e\in(0,\e_3).$ From here, the proof follows analogously to the proof of Theorem \ref{teo1} by applying Theorem \ref{l5}.
\end{proof}

\subsection{Lyapunov Coefficient}\label{SLC}
Suppose that the expressions defined in \eqref{TeABC} for $\CA_{j},$ $\CB_{j},$ and $\C_j,$ $l\leq j\leq k,$ are known. If the Taylor expansion \eqref{aue} for $e^{i \T_{\e}}$ is explicitly known, then we can also compute the following Taylor expansions around $\e=0,$ 
 \begin{align}\label{TeE}
 \begin{split}
&e^{-i\theta_\e}=1+\e^l \sum_{n=0}^{k-2l}\e^nr_{n+l}+\CO(\e^{k-l+1})\quad \mbox{and}  \\
&\dfrac{\e^l(1-2e^{i\theta_\e}) e^{-2 i\T_{\e}}}{1-e^{-i\theta_\e}}=\sum_{n=0}^{k-l}\e^n s_n+\CO(\e^{k-l+1}).
\end{split}
\end{align}
Notice that $s_0=-\dfrac{i}{ \omega_0}$ and $r_l=-i \omega_0.$

In what follows, we provide the formulae for $\ell_{1,j},$ $l\leq j\leq k.$ Thus, let ${\bf p}=(1,-i)/\sqrt{2}$ and, for $m\leq k-l,$ denote
\begin{align*}
L_{m}=&\left\langle {\bf p}, C_{l+m}({\bf  p}, {\bf  p}, {\bf \ov p})\right\rangle-\hspace{-0.45 cm}\sum_{n_1+n_2+n_3=m}\hspace{-0.45 cm} s_{n_1}\left\langle {\bf p}, B_{l+n_2}({\bf p}, {\bf p})\right\rangle\left\langle {\bf p}, B_{l+n_3}({\bf p}, \ov{\bf p})\right\rangle,\\
\widetilde L_{m}=&\left\langle {\bf p}, C_{l+m}({\bf  p}, {\bf  p}, {\bf \ov p})\right\rangle-\hspace{-0.45 cm}\sum_{n_1+n_2+n_3=m}\hspace{-0.45 cm} s_{n_1}\left\langle {\bf p}, B_{l+n_2}({\bf p}, {\bf p})\right\rangle\left\langle {\bf p}, B_{l+n_3}({\bf p}, \ov{\bf p})\right\rangle\\
&+\hspace{-0.45 cm}\sum_{n_1+n_2=m}\hspace{-0.45 cm}r_{l+n_1}\left\langle {\bf p}, C_{l+n_2}({\bf  p}, {\bf  p}, {\bf \ov p})\right\rangle-\hspace{-0.45 cm}\sum_{n_1+n_2=m}\hspace{-0.45 cm}\langle {\bf p}, B_{l+n_1}({\bf p}, \ov{\bf p})\rangle\langle B_{l+n_1}({\bf p}, \ov{\bf p}), {\bf p}\rangle\\
&+\dfrac{1}{2}\hspace{-0.45 cm}\sum_{n_1+n_2=m}\hspace{-0.45 cm}\langle {\bf p}, B_{l+n_1}(\ov {\bf p}, \ov{\bf p})\rangle\langle B_{l+n_1}(\ov{\bf p}, \ov{\bf p}), {\bf p}\rangle.
\end{align*}

\begin{proposition}\label{propform}Assume the hypotheses of Theorem \ref{teo2} and consider the coefficients $\ell_{1,j},$ $l\leq j\leq k,$ as defined in \eqref{Tel1}. Then, 
\begin{align*}
\ell_{1,l+m}=\left\{
	\begin{array}{ll}
		\dfrac{1}{2}\textrm{Re}\left( L_{m}\right) & \mbox{for } \, 0\leq m < l ,\\
		& \\
		\dfrac{1}{2}\textrm{Re}\left(L_{m}+\widetilde L_{m-l} -L_{m-l}\right) & \mbox{for } \, l\leq m \leq k-l.
	\end{array}
\right.
\end{align*}
\end{proposition}

\begin{proof}
Substituting \eqref{gijhigher} and \eqref{TeE} into \eqref{ell1} and collecting the coefficients of $\e^l$ and $\e^{2l},$ we have
\begin{equation}\label{ellK}
\ell_1=\e^l \dfrac{\textrm{Re}(K^1_{\e})}{2}+\e^{2l}\dfrac{\textrm{Re}(K^2_{\e})}{2}+\CO(\e^{k+1}),
\end{equation}
where
\begin{align*}
 \begin{split}
K^1_{\e}=&\left\langle {\bf p}, \CC_{\e}({\bf  p}, {\bf  p}, {\bf \ov p})\right\rangle)-\left\langle {\bf p}, \CB_{\e}({\bf p}, {\bf p})\right\rangle\left\langle {\bf p}, \CB_{\e}({\bf p}, \ov{\bf p})\right\rangle\sum_{n=0}^{k-l}\e^n s_n\\
K^2_{\e}=&\left\langle {\bf p}, \CC_{\e}({\bf  p}, {\bf  p}, {\bf \ov p})\right\rangle\sum_{n=0}^{k-2l}\e^nr_{l+n}-\left\langle {\bf p}, \CB_{\e}({\bf p},\ov {\bf p})\right\rangle\left\langle \CB_{\e}({\bf p}, \ov{\bf p}), {\bf p}\right\rangle\\
&-\dfrac{1}{2}\left\langle {\bf p}, \CB_{\e}({\bf p}, {\bf p})\right\rangle\left\langle  \CB_{\e}({\bf p}, {\bf p}),{\bf p}\right\rangle.
\end{split}
\end{align*}

First, notice that
\begin{align*}
\left\langle {\bf p}, \CB_{\e}({\bf q}, {\bf q})\right\rangle\left\langle {\bf p}, \CB_{\e}({\bf p}, \ov{\bf p})\right\rangle\sum_{n=0}^{k-l}\e^n s_n=& \sum^{k-l}_{m=0}\e^m\hspace{-0.45 cm}\sum_{n_1+n_2+n_3=m}\hspace{-0.45 cm} s_{n_1}\left\langle {\bf p}, B_{l+n_2}({\bf p}, {\bf p})\right\rangle\left\langle {\bf p}, B_{l+n_3}({\bf p}, \ov{\bf p})\right\rangle\\
&+\CO(\e^{k-l+1}).
\end{align*}
Thus, from \eqref{TeABC},
\begin{align}\label{K1}
K^1_{\e}=\sum_{m=0}^{k-l} \e^m L_m.
\end{align}

Now, since
\begin{align*}
\left\langle {\bf p}, \CC_{\e}({\bf  p}, {\bf  p}, {\bf \ov p})\right\rangle\sum_{n=0}^{k-2l}\e^nr_{l+n}&=\sum^{k-2l}_{m=0}\e^m\hspace{-0.45 cm}\sum_{n_1+n_2=m}\hspace{-0.45 cm}r_{l+n_1}\left\langle {\bf p}, C_{l+n_2}({\bf  p}, {\bf  p}, {\bf \ov p})\right\rangle+\CO(\e^{k-2l+1})\\
\left\langle {\bf p}, \CB_{\e}({\bf p},\ov {\bf p})\right\rangle\left\langle \CB_{\e}({\bf p}, \ov{\bf p}), {\bf p}\right\rangle&=\sum^{k-2l}_{m=0}\e^m\hspace{-0.45 cm}\sum_{n_1+n_2=m}\hspace{-0.45 cm}\langle {\bf p}, B_{l+n_1}({\bf p}, \ov{\bf p})\rangle\langle B_{l+n_1}({\bf p}, \ov{\bf p}), {\bf p}\rangle+\CO(\e^{k-2l+1}),\\
\left\langle {\bf p}, \CB_{\e}({\bf p}, {\bf p})\right\rangle\left\langle  \CB_{\e}({\bf p}, {\bf p}),{\bf p}\right\rangle&=\sum^{k-2l}_{m=0}\e^m\hspace{-0.45 cm}\sum_{n_1+n_2=m}\hspace{-0.45 cm}\langle {\bf p}, B_{l+n_1}(\ov {\bf p}, \ov{\bf p})\rangle\langle B_{l+n_1}(\ov{\bf p}, \ov{\bf p}), {\bf p}\rangle+\CO(\e^{k-2l+1}),
\end{align*}
we have that
\begin{align}\label{K2}
K^2_{\e}=\sum_{m=0}^{k-2l} \e^m(\widetilde L_{m}-L_m).
\end{align}

Substituting \eqref{K1} and \eqref{K2} into \eqref{ellK}, we get
\[
\ell_1=\dfrac{\e^l}{2}\left(\sum_{m=0}^{k-l}\e^m \textrm{Re}(L_m)+\sum_{m=l}^{k-l}\e^m  \textrm{Re}(\widetilde L_{m-l}-L_{m-l})\right)+\CO(\e^{k+1}).
\]
From here, we split the analysis in two cases, namely $k-l<l$ and $l\leq k-l.$

If $k-l<l,$ we have
\[
\ell_1=\sum_{m=0}^{k-l}\e^{m+l} \dfrac{\textrm{Re}(L_m)}{2}+\CO(\e^{k+1}).
\]
Thus, in this case, $\ell_{1,l+m}=\frac{1}{2}\textrm{Re}(L_m),$ for $m=0,1,\ldots,k-l<l.$

Otherwise, if $l\leq k-l,$ we have
\[
\ell_1=\sum_{m=0}^{l-1}\e^{m+l} \dfrac{\textrm{Re}(L_m)}{2}+\sum_{m=l}^{k-l}\e^{m+l}  \textrm{Re}(L_m+\widetilde L_{m-l}-L_{m-l})+\CO(\e^{k+1}).
\]
Thus, in this case,
\begin{align*}
\ell_{1,l+m}=\left\{
	\begin{array}{ll}
		\dfrac{1}{2}\textrm{Re}\left( L_{m}\right) & \mbox{for } \, 0\leq m < l ,\\
		& \\
		\dfrac{1}{2}\textrm{Re}\left(L_m+\widetilde L_{m-l}-L_{m-l}\right) & \mbox{for } \, l\leq m \leq k-l.
	\end{array}
\right.
\end{align*}

Hence, we have concluded this proof.
\end{proof}

\subsection{Hopf bifurcation in the $l$th averaged system}\label{HPl}  Let $l,$ $1\leq l<k,$ be the subindex of the first non-vanishing averaging function. As set in hypothesis {\bf B1}, $(\bx_{\mu_0};\mu_0)$ is a Hopf point of the averaged system of order $l$ \eqref{tasl}. It is easy to see that its first Lyapunov coefficient at  $(\bx_{\mu_0};\mu_0)$ is given by $\e^l \,\ell_{1,l}.$ In addition, from Proposition \ref{propform}, we compute
\[
\begin{array}{rl}
\ell_{1,l}&=\dfrac{1}{8}\left(\dfrac{\partial ^3\bg_l^1(\bx_{\mu_0};\mu_0)}{\partial x^3}+\dfrac{\partial ^3\bg_l^1(\bx_{\mu_0};\mu_0)}{\partial x\partial y^2}+\dfrac{\partial ^3\bg_l^2(\bx_{\mu_0};\mu_0)}{\partial x^2\partial y}+\dfrac{\partial ^3\bg_l^2(\bx_{\mu_0};\mu_0)}{\partial y^3}\right)\\
&+\dfrac{1}{8\omega_0}\left(\dfrac{\partial ^2\bg_l^1(\bx_{\mu_0};\mu_0)}{\partial x\partial y}\Big(\dfrac{\partial ^2\bg_l^1(\bx_{\mu_0};\mu_0)}{\partial x^2}+\dfrac{\partial ^2\bg_l^1(\bx_{\mu_0};\mu_0)}{\partial y^2}\Big)-\dfrac{\partial ^2\bg_l^2(\bx_{\mu_0};\mu_0)}{\partial x\partial y}\right.\\
&\Big(\dfrac{\partial ^2\bg_l^2(\bx_{\mu_0};\mu_0)}{\partial x^2}+\dfrac{\partial ^2\bg_l^2(\bx_{\mu_0};\mu_0)}{\partial y^2}\Big)-\dfrac{\partial ^2\bg_l^1(\bx_{\mu_0};\mu_0)}{\partial x^2}\dfrac{\partial ^2\bg_l^2(\bx_{\mu_0};\mu_0)}{\partial x^2}\\
&\left.+\dfrac{\partial ^2\bg_l^1(\bx_{\mu_0};\mu_0)}{\partial y^2}\dfrac{\partial ^2\bg_l^2(\bx_{\mu_0};\mu_0)}{\partial y^2}\right).
\end{array}
\]
We notice that a Hopf bifurcation in the averaged system \eqref{tasl} is characterized by hypotheses {\bf B1}, {\bf B2}, and $\ell_{1,l}\neq0.$ Hypothesis {\bf B3} is just a technical assumption with no loss of generality. Therefore, from Theorem \ref{teo2}, a Hopf bifurcation in the averaged system of order $l$ \eqref{tasl} implies in a Torus bifurcation in the differential equation \eqref{s1}.

\section{Invariant torus in a 3D vector field}\label{Ex}
In this section, as an example of application of the developed theory, we show how to use Theorems \ref{teo1} and \ref{teo2} for detecting an invariant torus in the following  family of 3D vector fields,
\begin{equation}\label{exvf}
\begin{array}{rl}
x'=&-y+\e P_1(x,y,z;\mu)+\e^2 P_2(x,y,z;\mu)+\CO(\e^3),\vspace{0.1cm}\\
y'=&x+\e^2 Q(x,y,z;\mu)+\CO(\e^3),\vspace{0.1cm}\\
z'=&\e R_1(x,y,z;\mu)+\e^2 R_2(x,y,z;\mu)+\CO(\e^3),
\end{array}
\end{equation}
where 
\[
\begin{array}{rl}
P_1(x,y,z;\mu)=&10 x (3\mu+a\,(9+4x^2+3z^2))-30x\rho(x,y)\Big(z+\mu+a\,(1+4x^2+z^2)\Big),\vspace{0.3cm}\\
P_2(x,y,z;\mu)=&10b\,x(9+4x^2+3z^2)-30b\,x\rho(x,y)\Big(1+4x^2+z^2\Big),\vspace{0.3cm}\\
Q(x,y,z;\mu)=&-30 \pi  y(15(\mu ^2-1)+20 a\,(\mu (4 y^2+3 z^2+9)+3 z))\vspace{0.2cm}\\
&+3 a^2(24 y^4+40 y^2(z^2+5)+15(z^4+6 z^2+5))\vspace{0.2cm}\\
&+150 \pi  y \rho(x,y)\Big(3 \mu ^2+6 \mu  (2 a\, (4 y^2+z^2+1)+z)\vspace{0.2cm}\\
&+a\, (3 a\, (24 y^4+8 y^2 (3 z^2+5)+3 (z^2+1)^2)+8 y^2 z+6 z^3+6 z)-3\Big),\vspace{0.3cm}\\
\end{array}
\]
\[
\begin{array}{rl}
R_1(x,y,x;\mu)=&30 x^2 (1-2 a\, z) \rho (x,y)+15 (a\, (2 x^2 z+z^3+z)+\mu  z-1),\vspace{0.3cm}\\
R_2(x,y,z;\mu)=&-225 \pi  a^2 z (8 y^4+12 y^2 (z^2+3)+3 (z^2+1)^2)\vspace{0.2cm}\\
&-450 \pi  a\, (y^2 (4 \mu  z-6)+(z^2+1) (2 \mu  z-1))+15 b\, (2 x^2 z+z^3+z)\vspace{0.2cm}\\
&-225 \pi  ((\mu ^2-1) z-2 \mu )+60\rho(x,y) \Big(5 \pi  y^2 (a\, ((6 a\, z-1) (4 y^2+3 z^2)\vspace{0.2cm}\\
&+18 a\, z+12 \mu  z-9)-3 \mu )-b\, x^2 z\Big),
\end{array}
\]
and 
\[
\rho(x,r)=\dfrac{1}{\sqrt{x^2+y^2}}.
\]

\begin{proposition}\label{p1}
For $\e>0$ and $|\mu|$ sufficiently small the vector field \eqref{exvf} admits a unique limit cycle $\f(t;\mu,\e)=(x(t;\mu,\e),y(t;\mu,\e),z(t;\mu,\e))$  satisfying $x(t;\mu,0)^2+y(t;\mu,0)^2=1$ and $z(t;\mu,0)=0,$ for every $t\in\R.$ Assume that $a^2+b^2\neq0.$ Then,  there exists a smooth curve $\mu(\e),$ defined for $\e>0$ sufficiently small and satisfying $\mu(\e)= -\e \pi/2+\CO(\e^2),$ and neighborhoods $\mathcal{V}_{\e}\subset \R^3$ of the limit cycle $\f(t;\mu(\e),\e)$ and $J_{\e}$ of $\mu(\e)$ for which the following statement holds.
\begin{itemize}
 \item[(i)] If $a\neq0,$ then a unique invariant torus of the vector field \eqref{exvf} bifurcates in $\mathcal{V}_{\e}$ from the limit cycle $\f(t;\mu(\e),\e),$ as $\mu$ passes through $\mu(\e).$ Such a torus exists whenever $\mu\in J_{\e}$ and $a(\mu-\mu(\e))<0$ and surrounds the limit cycle $\f(t;\mu,\e).$ In addition,  if $a>0$ (resp. $a<0$) the torus is unstable (resp. stable), whereas the limit cycle $\f(t;\mu,\e)$ is asymptotically stable (resp. unstable) (see Figure \ref{fig1}).

\smallskip

 \item[(ii)]  If $a=0$ and $b\neq0,$ then a unique invariant torus of the vector field \eqref{exvf}  bifurcates in $\mathcal{V}_{\e}$ from the limit cycle $\f(t;\mu(\e),\e),$ as $\mu$ passes through $\mu(\e).$ Such a torus exists whenever $\mu\in J_{\e}$ and $b(\mu-\mu(\e))<0$ and surrounds the limit cycle $\f(t;\mu,\e).$ In addition,  if $b>0$ (resp. $b<0$) the torus is unstable (resp. stable), whereas the limit cycle $\f(t;\mu,\e)$ is stable (resp. unstable) (see Figure \ref{fig2}).
 \end{itemize}
\end{proposition}

\begin{proof}
Writing the vector field \eqref{exvf} in cylindrical coordinates $(x,y,z)=(r\cos\T,r \sin\T, z)$ and taking $\T$ as the new independent variable we get the following nonautonomous differential equation
\begin{equation}\label{excil}
\begin{array}{l}
\dfrac{\p r}{\p \T}=\e\cos\T\widetilde P_1(\T,r,z)+\e^2\Big(\dfrac{1}{r} \cos\T\sin\T\widetilde P_1(\T,r,z)^2+ \cos\T\widetilde P_2(\T,r,z)+\sin\T\widetilde Q(\T,r,z)\Big),\vspace{0.3cm}\\
\dfrac{\p z}{\p\T}=\e\widetilde R_1(\T,r,z)+\e^2\Big(\dfrac{1}{r}\sin\T\widetilde P_1(\T,r,z)\widetilde R_1(\T,r,z)+ \widetilde R_2(\T,r,z)\Big),
\end{array}
\end{equation}
where $\widetilde P_i(\T,r,z)=P_i(r\cos\T,r\sin\T,z)$ and $\widetilde R_i(\T,r,z)=R_i(r\cos\T,r\sin\T,z),$ for $i=1,2,$ and $\widetilde Q(\T,r,z)=Q(r\cos\T,r\sin\T,z).$ Computing the first and second averaging functions of  \eqref{excil} we get
\[
\begin{array}{rl}
\bg_1(r,z)=&30\pi\Big((r-1)\mu-z+a\,(r-1)((r-1)^2+z^2)\,,\,r-1+\mu z\\
&+a\,z((r-1)^2+z^2)\Big),\vspace{0.3cm}\\
\bg_2(r,z)=&30\pi b\,\Big((r-1)((r-1)^2+z^2)\,,\,z((r-1)^2+z^2)\Big).
\end{array}
\]
We notice that $\bg_1$ satisfies hypothesis {\bf B1} for $\bx_{\mu}=(1,0)$ and $\mu_0=0.$

Following the method described in the previous section we take $ \by= \bx+\bxi(\mu,\e)$ and $\mu=\sigma+\mu(\e).$ It is easy to see that $\bxi(\mu, \e)=(1,0)+\CO(\e^2)$  and $\mu(\e)= -\e \pi/2+\CO(\e^2).$ Thus, the transformed Poincar\'{e} map \eqref{sish2} for the differential equation  \eqref{excil} writes 
$$
 {\bf H}_\e(\by,0)= A_\e\by+\dfrac{1}{2}B_\e(\by,\by)+\dfrac{1}{6} C_\e(\by,\by,\by)+\CO(|\by|^4),
$$
with $A_\e=Id+\e^1 A_1+\e^{2} A_{2}+\CO(\e^{3}),$ $B_\e({\bf u},{\bf v})=\e^1 B_1({\bf u},{\bf v})+\e^{2} B_{2}({\bf u},{\bf v})+\CO(\e^{3}),$ and $C_\e({\bf u},{\bf v},{\bf w})=\e^1 C_1({\bf u},{\bf v},{\bf w})+\e^{2} C_{2}({\bf u},{\bf v},{\bf w})+\CO(\e^{3}),$ where  
$$A_1=\begin{pmatrix}
0 & -1\\
1 & 0
\end{pmatrix}, \,
A_2=\begin{pmatrix}
-\dfrac{1}{2} & 0\\
0 & -\dfrac{1}{2}
\end{pmatrix},B_1({\bf u},{\bf v})=B_2({\bf u},{\bf v})=(0,0),$$
  and 
  \[
  \begin{array}{rl}
  C_{i}({\bf u},{\bf v},{\bf w})=& 2\,\ell_1^i\Big(3 u_1 v_1 w_1 + u_2 v_2 w_1 + u_2 v_1 w_2 + u_1 v_2 w_2, \\
  &u_2 v_1 w_1 + u_1 v_2 w_1 + u_1 v_1 w_2 + 3 u_2 v_2 w_2 \Big),
  \end{array}
  \]
   for $i=1,2.$ 

Notice that  $Id+\e A_1+\e^2 A_2$ satisfies hypotheses {\bf B3}. Thus, we can define $e^{i \theta_{\e}}=1+\e \,i -\e^2\,\dfrac{1}{2}+\CO(\e^3).$  From Proposition \eqref{propform} and \eqref{a0e} we get
$
  \ell_1=\e \, 4\,a+\e^24\,b+\CO(\e^3).
$ 
In this case, following the notation \eqref{Tel1}, $\ell_{1,1}=4a\,$ and $\ell_{1,2}=4 b\,.$
Hence, applying Theorem \ref{teo2} we conclude this proof.
\end{proof}

\begin{figure}[H]
	\subfigure[]
	{\begin{overpic}[width=6cm]{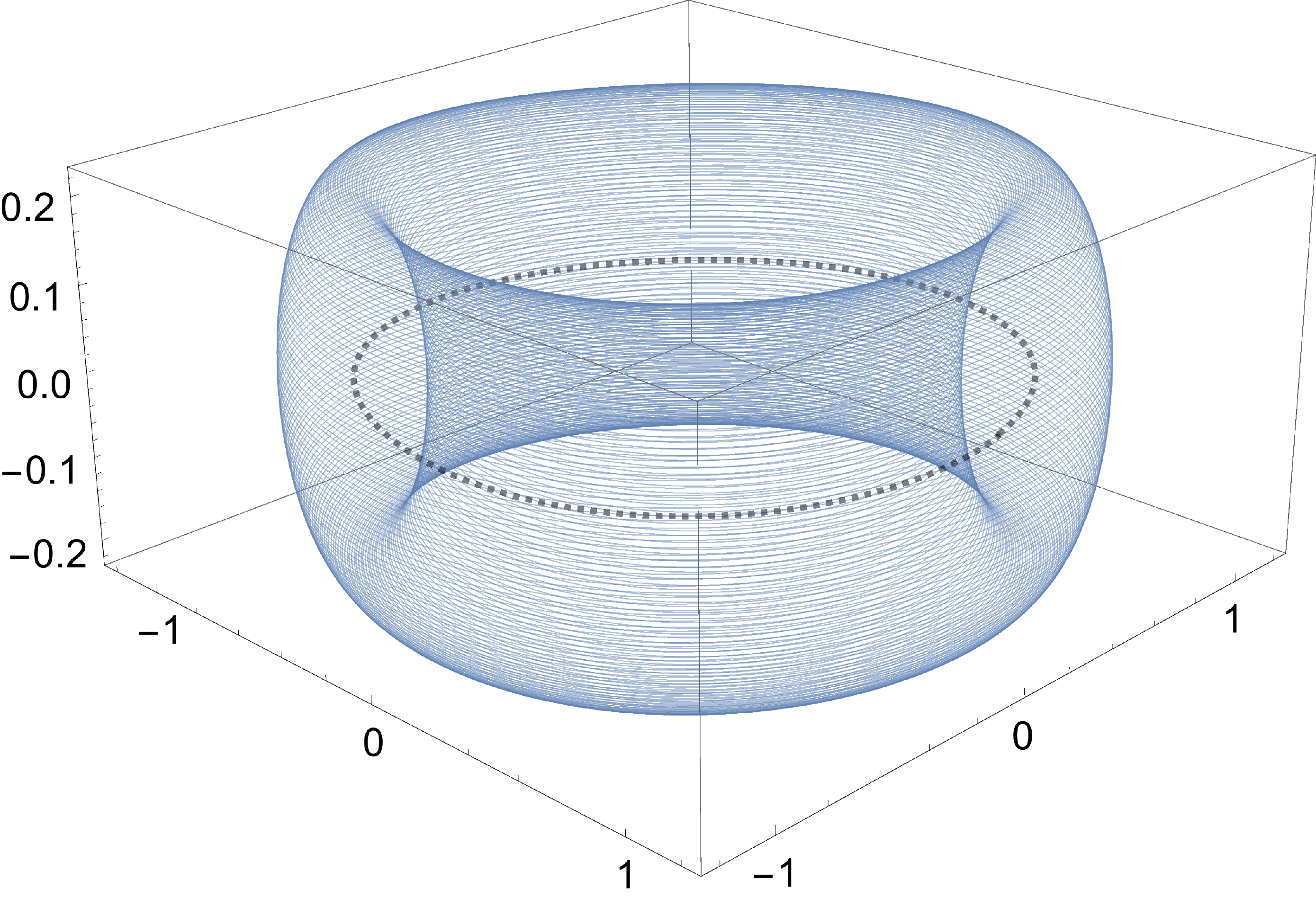}
	%\begin{overpic}[grid,tics=10,width=6cm]{Toro1.pdf}
			\put(26,7){$x$}
			\put(79,10){${y}$}
			\put(-3,37.5){${z}$}
	\end{overpic}}
	%\hspace*{0.5cm}
	\subfigure[]
	{\begin{overpic}[width=6cm]{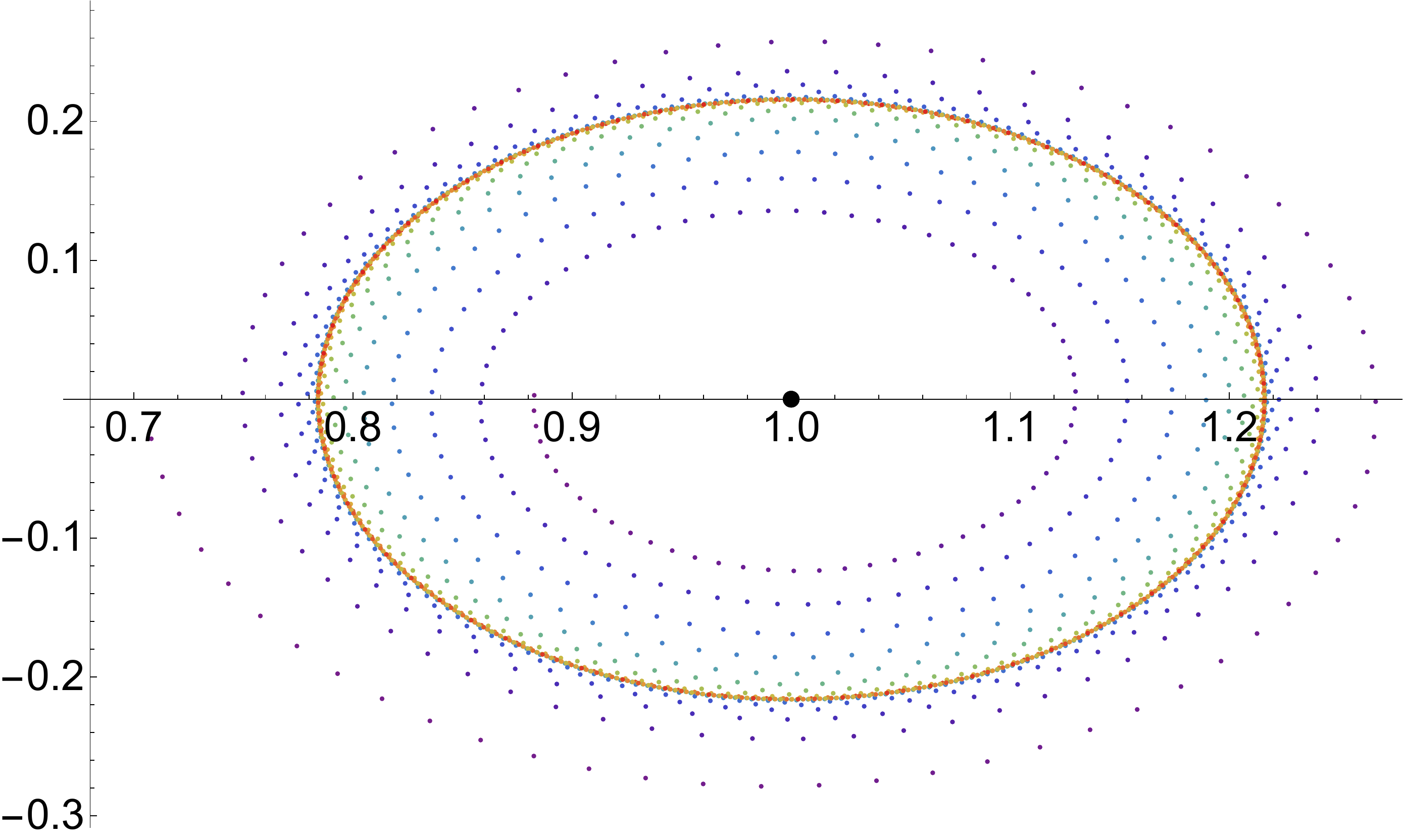}
	%\begin{overpic}[grid,tics=10,width=6cm]{figura2.pdf}
				\put(5,62){$z$}
				\put(100,30){$x$}
	\end{overpic}}
	\caption{Invariant torus of the vector field \eqref{exvf} predicted by Proposition \ref{p1} (i) assuming $a\,=-1,$ $ b\,=0,$ $\mu=0,$ and $\e=10^{-3}.$ Figure (a) shows a trajectory starting at $(0.98, 0.21, 0),$ for $t\in [50000,51250].$ The dashed line corresponds to the limit cycle. Figure (b) depicts the Poincar\'{e} section $y=0, x>0$ showing the stable invariant torus as a stable invariant closed curve.}
	\label{fig1}	
\end{figure}

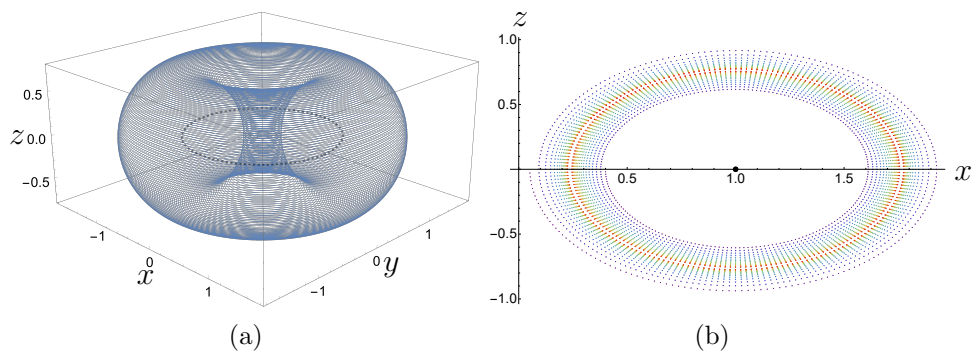
\begin{figure}[H]
	\subfigure[]
	{\begin{overpic}[width=6cm]{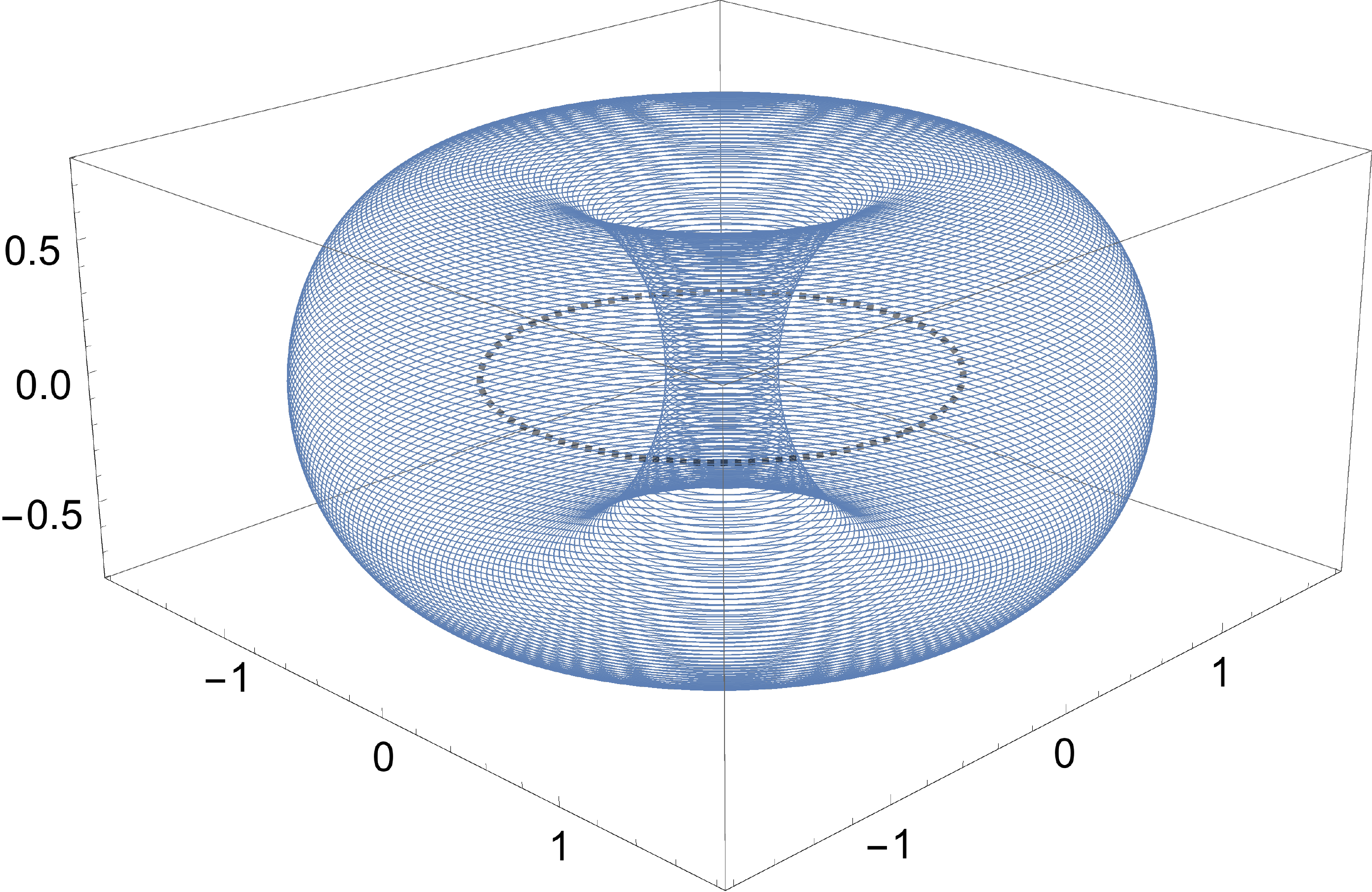}
	%\begin{overpic}[grid,tics=10,width=6cm]{Toro2.pdf}
			\put(25,5){$x$}
			\put(79,8){${y}$}
			\put(-3,36){${z}$}
	\end{overpic}}
	%\hspace*{0.5cm}
	\subfigure[]
	{\begin{overpic}[width=6cm]{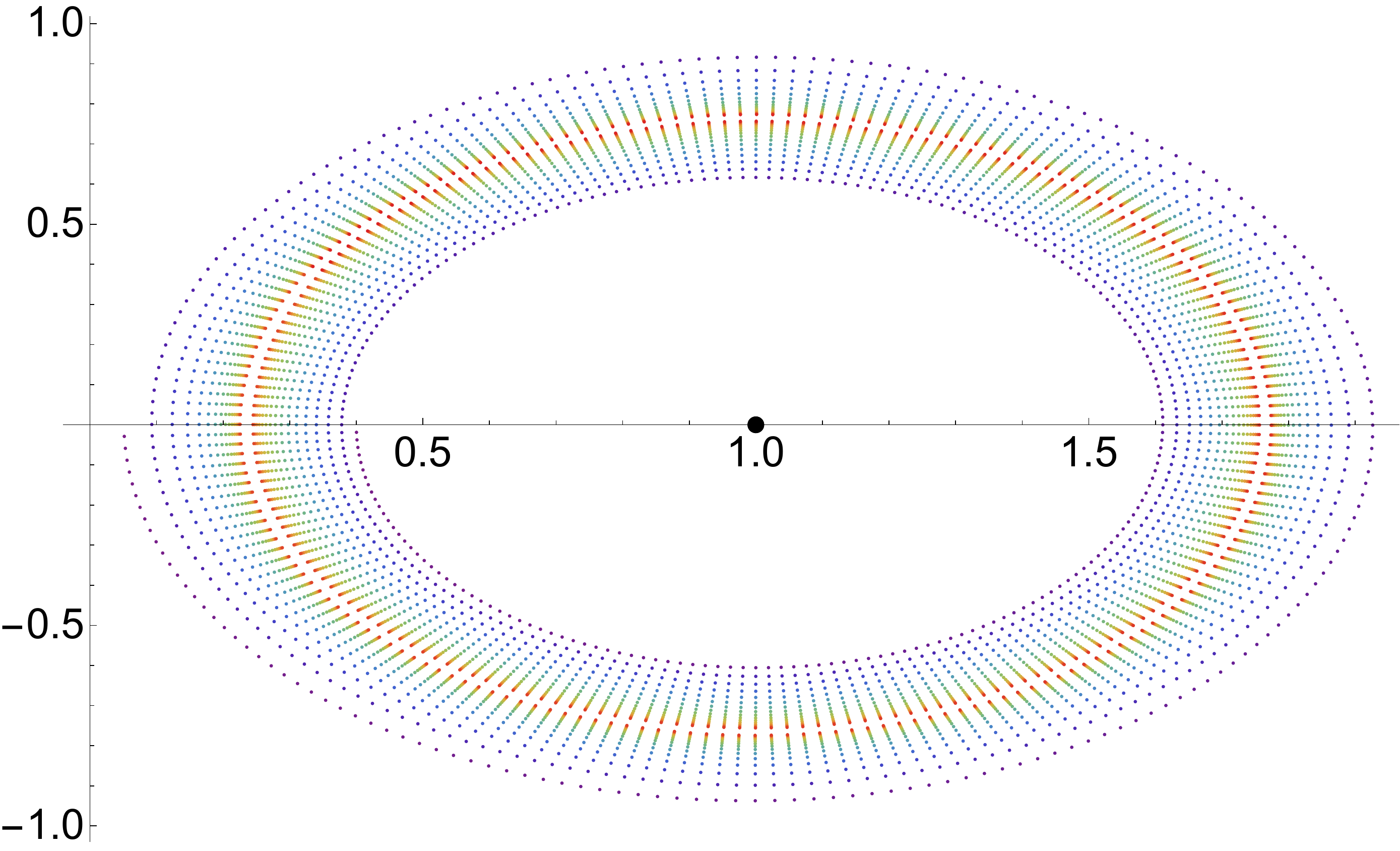}
	%\begin{overpic}[grid,tics=10,width=6cm]{figura22.pdf}
				\put(5,62){$z$}
				\put(102,28){$x$}
	\end{overpic}}
	
	\caption{Invariant torus of the vector field \eqref{exvf} predicted by Proposition \ref{p1} (ii) assuming $a\,=0,$ $ b\,=-80,$ $\mu=0,$ and $\e=10^{-3}/3.$ The dashed line corresponds to the limit cycle. Figure (a) shows a trajectory starting at $(0.385, 0, 0.386),$ for $t\in [57540,59000].$ Figure (b) depicts the Poincar\'{e} section $y=0, x>0$ showing the stable invariant torus as a stable invariant closed curve.}
	\label{fig2}	
\end{figure}

\section*{Appendix: Higher order averaged functions}

The averaging theory is one of the most classical analytical methods to study isolated periodic solutions of differential equations in the presence of a small parameter. Usually, this theory deals with differential equations in the following standard form
\begin{equation}\label{ss1}
\dot\bx(t)=\sum_{i=1}^k\e^i \bF_i(t,\bx)+\e^{k+1} \widetilde{\bF}(t,\bx,\e),
\end{equation}
where $\bF_i$ and $\widetilde\bF$ are  sufficiently smooth functions, $T$-periodic in the variable $t,$  $\bx\in\Omega$ with $\Omega$ an open bounded subset of $\mathbb{R}^2,$ $t \in \R,$ and $\e \in
(-\e_0,\e_0)$ for some $\e_0>0$ small. In \cite{LNT,novaes} it has been established that the {\it $i$-th order averaged function} of \eqref{ss1} is given by
\begin{equation*}\label{f}
\bg_i(\bx)=\dfrac{y_i(T,\bx)}{i!},
\end{equation*}
where $y_i:\R\times D\rightarrow \R^n,$ for $i=1,2,\ldots,k,$ are
defined recurrently as
\begin{equation*}\label{ynew}
\begin{array}{rl}
y_1(t,\bx)=&\displaystyle\int_0^tF_1\left(s,\bx\right) \,ds,\vspace{0.3cm}\\
y_i(t,\bx)=&\displaystyle \int_0^t\Bigg(i! F_i\left(s,\bx\right)\vspace{0.2cm}\\
&\displaystyle+\sum_{l=1}^{i-1}\sum_{m=1}^l\dfrac{i!}{l!}\p^m F_{i-l} \left(s,\bx\right) \mathbb{B}_{l,m}\big(y_1(s,\bx),\ldots,y_{l-m+1}(s,\bx)\big)\Bigg)ds.
\end{array}
\end{equation*}
Here, $\p^LF(t,\bx)$ denotes the Frechet's derivative with respect to the variable $\bx,$ which is a $L$-multilinear map applied to a ``product'' of
$L$ vectors of $\R^n,$ $\bigodot_{j=1}^Ly_j\in
\R^{nL},$ where   $y_j= (y_{j1},\ldots, y_{jn})\in \R^n.$ Formally,
\begin{equation*}\label{p}
\p^L F(t,\bx)\bigodot_{j=1}^Ly_j= \sum_{i_1,\ldots,i_L=1}^n \dfrac{\p^L
F(t,\bx)}{\p x_{i_1}\cdots \p x_{i_L}}y_{1i_1}\cdots y_{Li_L}.
\end{equation*}
Also, for $p$ and $q$ positive integers, $\mathbb{B}_{p,q}$ denotes the partial Bell polynomials,
\[
\mathbb{B}_{p,q}(x_1,\ldots,x_{p-q+1})=\sum_{\widetilde S_{p,q}}\dfrac{p!}{b_1!\,b_2!\cdots b_{p-q+1}!}\prod_{j=1}^{p-q+1}\left(\dfrac{x_j}{j!}\right)^{b_j},
\]
where now $\widetilde{S}_{p,q}$ is the set of all $(p-q+1)$-tuple of nonnegative integers $(b_1,b_2,\cdots,b_{p-q+1})$ satisfying $b_1+2b_2+\cdots+(p-q+1)b_{p-q+1}=p,$ and
$b_1+b_2+\cdots+b_{p-q+1}=q.$ 

The next results were proved in \cite{LNT}. Nonsmooth versions of these results can be found in \cite{llinovrod}.

\begin{lemma}[\cite{LNT}]\label{lap}
Let $\f(\cdot,\bx,\e):[0,t_{\e}]\rightarrow\R^n$ be the solution of
\eqref{s1} with $\f(0,\bx,\e)=\bx.$ Then, for $|\e|$ sufficiently small, $t_{\e}>T$ and
\[
\f(t,\bx,\e)=\bx+\sum_{i=1}^{k}\e^{i}\dfrac{y_i(t,\bx)}{i!}+\CO(\e^{k+1}).
\]
\end{lemma}

\begin{theorem}[\cite{LNT}]\label{tap}
Assume that, for some $l\in \{1,2,\ldots, k\},$ $\bg_i= 0$ for $i=1,2,\ldots,l-1,$ and
$\bg_l\neq 0.$
Then, for each $a^*\in D$ such that $\bg_l(a^*)=0$ and $\det(df_l(a^*))\neq0,$ there exists, for $|\e|>0$ sufficiently small, a $T$-periodic solution $\f(\cdot,\e)$ of \eqref{ss1} such that $\f(0,\e)\to a^*$ when
$\e\to 0.$
\end{theorem}

\section*{Acknowledgements}
We thank to the referee for the helpful comments and suggestions.

The authors thank Espa\c{c}o da Escrita -- Pr\'{o}-Reitoria de Pesquisa -- UNICAMP for the language services provided.

MRC is partially supported by Brazilian FAPESP grants 2018/07344-0 and 2019/05657-4. DDN is partially supported by a Brazilian FAPESP grant 2018/16430-8 and by Brazilian CNPq grants 306649/2018-7 and 438975/2018-9. 

\bibliographystyle{abbrv}
\bibliography{gvd}
\end{document}